\documentclass[a4paper,reqno,oneside]{amsart}

\usepackage{a4wide}
\usepackage{hyperref}
\usepackage{amsmath,amsfonts,amssymb,amsthm}
\usepackage[utf8]{inputenc}
\usepackage{ifthen}
\usepackage{color}
\usepackage{graphicx}
\usepackage{todonotes}

\usepackage{enumitem}

\newtheorem{theorem}{Theorem}[section]

\newtheorem{definition}{Definition}
\newtheorem{lemma}[theorem]{Lemma}
\newtheorem{remark}[theorem]{Remark}

\usepackage{todonotes}

\usetikzlibrary{arrows,shapes}
\tikzstyle{vertex}=[circle,fill=black!15,minimum size=12pt,inner sep=0pt,
font=\footnotesize, text=black!50]
\tikzstyle{selected vertex} =
[vertex, fill=black!100,text=white, font=\bf]
\tikzstyle{selected point} =
[vertex, fill=black!100, minimum size = 4pt]
\tikzstyle{point} =
[vertex, fill=black!40, minimum size = 4pt]
\tikzstyle{edge} = [draw,black!40,-]
\tikzstyle{selected edge} = [draw,thick,-]
\tikzstyle{strong edge} = [draw,ultra thick,-]
\tikzstyle{weight} = [font=\small]
\tikzstyle{ignored edge} = [draw,line width=5pt,-,black!20]

\begin{document}

%\title[Avoiding subgraphs in series-parallel graphs]{Avoiding subgraphs in series-parallel graphs}

\title{Subgraph statistics in subcritical graph classes}

\author{Michael Drmota}
\address{(MD) TU Wien, Institute of Discrete Mathematics and Geometry, 1040 Wien, Austria}
\email{michael.drmota@tuwien.ac.at}
\urladdr{http://www.dmg.tuwien.ac.at/drmota/}

\author{Lander Ramos}
\address{(LR) Universitat Polit\`ecnica de Catalunya. Departament de Matem\`atica aplicada $2$, 08034 Barcelona, Spain}
\email{lander.ramos@upc.edu}
\urladdr{http://www-ma2.upc.edu/lramos/}

\author{Juanjo Ru\'e}
\address{(JR) Freie Universit\"at Berlin, Institut f\"ur Mathematik und Informatik, 14195 Berlin, Germany}
\email{jrue@zedat.fu-berlin.de}
\urladdr{http://www-ma2.upc.edu/jrue/}

\thanks{
M.D.~was supported by the SFB F50 \emph{Algorithmic and Enumerative Combinatorics} of the
Austria Science Foundation FWF.
J.\,R.~was partially supported by the FP7-PEOPLE-2013-CIG project CountGraph (ref. 630749), the Spanish MICINN projects MTM2014-54745-P and MTM2014-56350-P, the DFG within the Research Training Group \emph{Methods for Discrete Structures} (ref. GRK1408), and the \emph{Berlin Mathematical School}.
}
\maketitle

%%%%%%%%%%%%%%%%%%%%%%%%%%%%%%%%%%%%%%%%%%%%%%%%%%%%%%%%%%%%%%%%%%%%%%%%%%%%%%%%%%%%%%%%%%%%%%%%%%%%%%%%%%%%%%%%%%%%%%%%%%%%%%%%%%%%%%%%%%%%%%%%%%%%
%%%%%%%%%%%%%%%%%%%%%%%%%%%%%%%%%%%%%%%%%%%%%%%%%%%%%%%%%%%%%%%%%%%%%%%%%%%%%%%%%%%%%%%%%%%%%%%%%%%%%%%%%%%%%%%%%%%%%%%%%%%%%%%%%%%%%%%%%%%%%%%%%%%%
%%%%%%%%%%%%%%%%%%%%%%%%%%%%%%%%%%%%%%%%%%%%%%%%%%%%%%%%%%%%%%%%%%%%%%%%%%%%%%%%%%%%%%%%%%%%%%%%%%%%%%%%%%%%%%%%%%%%%%%%%%%%%%%%%%%%%%%%%%%%%%%%%%%%

\begin{abstract}
Let $H$ be a fixed graph and  $\mathcal{G}$ a subcritical graph
class.
In this paper we show that the number of occurrences of $H$ (as a subgraph) in a uniformly at random graph of size $n$ in $\mathcal{G}$ follows a normal limiting
distribution with linear expectation and variance.
The main ingredient in our proof is the analytic framework developed by Drmota, Gittenberger and Morgenbesser to deal with infinite systems of functional equations \cite{DrmGittMor2015}.
As a case study, we get explicit expressions for the number of triangles and cycles of length four for the family of series-parallel graphs.
\end{abstract}

%%%%%%%%%%%%%%%%%%%%%%%%%%%%%%%%%INTRODUCTION%%%%%%%%%%%%%%%%%%%%%%%%%%%%%%%%%%%%%%%%%%%%%%%%%%%%%%%%%%%%%%%%%%%%%%%%%%%%%%%%%%%%%%%%%%%%%%%%%%%%%%%
%%%%%%%%%%%%%%%%%%%%%%%%%%%%%%%%%%%%%%%%%%%%%%%%%%%%%%%%%%%%%%%%%%%%%%%%%%%%%%%%%%%%%%%%%%%%%%%%%%%%%%%%%%%%%%%%%%%%%%%%%%%%%%%%%%%%%%%%%%%%%%%%%%%%

\section{Introduction} \label{sec: introduction}

The study of subgraphs in random discrete structures is a central area in graph theory, which dates back to the seminal works of Erd\H{o}s and R\'enyi in the sixties \cite{ErRe60}.
Since then, lot of effort has been devoted to locate the threshold function for the appearance of a given subgraph in the $G(n,p)$ model, as well as the limiting distribution of the corresponding counting random variable (see for instance \cite{KaRu81,JaLuRu90, Ruc90}, and the monograph \cite[Chapter 3]{RanGraphs}).
The number of appearances of a fixed graph and its statistics had been also addressed as well in different restricted graph classes, including random regular graphs and random graphs with specified vertex degree (see for instance, \cite{McKay81,GaWo08, colored, KiSuVu07, Mckaydeg}, see also~\cite{McKayICM}) and random planar maps ~\cite{GaoWor03, GaWo04}.

In this paper we study subgraphs on a random graph in a so-called subcritical class.
Roughly speaking, a graph class is called \emph{subcritical} if the largest block of a random graph in the class with $n$ vertices has $O(\log(n))$ vertices
(see the precise analytic definition in Section 3).
Indeed, graphs in these classes have typically a tree-like structure and share several properties with trees.
Just to mention some families, prominent subcritical graph classes are forests, cacti trees, outerplanar graphs and series-parallel graphs, and more generally graph families defined by a finite set of $3$-connected components (see \cite{3-con}).
Let us mention that the analysis of subcritical graph classes is intimately related to the study of the random planar graph model: it is conjectured that a graph class defined by a set of excluded minors is subcritical if and only if at least one of the excluded graphs is
planar (see \cite{Noy15}).

The systematic study of subcritical graph classes started in \cite{Bernasconi2009} when studying the expected number of vertices of given degree. Later, in \cite{DrFuKaKrRu11} the authors extended the analysis to unlabelled graph classes, and obtained normal limiting probability distributions for different parameters, including the number of cut-vertices, blocks, edges and the vertex degree distribution.
Drmota and Noy~\cite{DrmotaNoyAnalco} investigated several extremal parameters in these graph classes.
They showed, for instance, that the expected diameter $D_n$ of a random connected graph on a subcritical graph class on $n$ vertices satisfies $c_1 \sqrt{n} \leq \mathbb{E}[D_n] \leq c_2 \sqrt{n \log n}$ for some constants $c_1$ and $c_2$.
More recently, the precise asymptotic estimate has been deduce to be of order $\Theta(\sqrt{n})$~\cite{panagiotou2015}.
Furthermore, the normalized metric space $(V(G), d_{G}/\sqrt{n})$ (where $d_G(u, v)$ denotes the number of edges in a shortest path that contains $u$ and $v$ in $G$) is shown to  converge with respect to the Gromov-Hausdorff metric to the so-called \emph{Brownian Continuum Random Tree} multiplied by an scaling factor that depends only the class under study (see~\cite{panagiotou2015} for details, and also \cite{stufler2015} for extensions to the unlabelled setting).
Let us also mention that even more recently, the Schramm-Benjamini convergence had been addressed as well in \cite{schben15,stufler2015} for these graph families.
Finally, the maximum degree and the degree sequence of a random series-parallel graph have been studied in \cite{Drmota2011} and \cite{Bernasconi2009, Drmota20102}, respectively.
\\
\paragraph{\textbf{Our results}:} this paper is a contribution to the understanding of the shape of a random graph from on these graph classes. More precisely, we present a very general framework to deal with subgraph statistics in subcritical graph classes. Our main result is the following theorem:

\begin{theorem}\label{thm:main}
Let $\mathcal{C}_n$ be the set of connected graphs of size $n$ in a subcritical graph class $\mathcal{G}$, and let $H$ be a fixed (connected) graph. Let $X_n^{H}$ be the number of copies of $H$ on a uniformly at random object in $\mathcal{C}_n$. Then,
$$\mathbb{E}[X_n^H]=\mu_{H} n+ O(1)\,\,and\,\, \mathbb{V}\mathrm{ar}[X_n^H]=\sigma^2_H n+O(1)$$
for some constants $\mu_H,\,\sigma^2_H$ that only depends on $H$ (and on the subcritical graph class under study). Moreover, if $\sigma_H^2>0$, then
$$\frac{X_n^{H}-\mathbb{E}[X_n^H]}{\sqrt{\mathbb{V}\mathrm{ar}[X_n^H]}}\rightarrow \mathrm{N(0,1)}.$$
\end{theorem}

The strategy we use on the proof is based on analytic combinatorics.
More precisely, given a subgraph $H$ we are able to get expressions for the counting formulas encoding the number of copies of $H$.
As we will show, even if $H$ has a very simple structure, we will need infinitely many equations and infinitely many variables to encode all the possible appearances.
Later, we will be able to fully analyze the infinite system of equations that we obtain using an adapted version of the main theorem of Drmota, Gittenberger and Morgenbesser \cite{DrmGittMor2015}, which provide the necessary analytic ingredient in order to study infinite functional systems of equations.
This result extends the classical Drmota-Lalley-Woods theorem for (finite) systems of functional equations (see for instance \cite{drmota97systems}).

Let us also discuss some similar results from the literature.
The study of induced subgraphs (also called patterns) in random trees was done in (\cite{ChyDrmKlaKo08}), showing normal limiting distributions with linear expectation and variance.
This covers in particular the distribution of the number of vertices of given degree in random trees.
In the more general setting of subcritical graph classes, the number of vertices of degree $k$ was studied in \cite{DrFuKaKrRu11}.
In another direction, appearances of a fixed subgraph $H$ (also called \emph{pendant copies}) in a subcritical graph class where studied in \cite{3-con} (see the proper definition of \emph{appearance} in~\cite{McStWe05}), showing again normal limiting distributions with linear expectation. As every appearance define a subgraph, this result shows that the number of subgraphs in a uniformly at random subcritical graph is (at least) linear. Our result strongly strengths this fact showing the precise limiting probability distribution.

As a case study, we get explicit constants for series-parallel graphs and for specific subgraphs. Recall that a graph is series-parallel if it excludes $K_4$ as a minor. Equivalently, a series-parallel graph has treewidth at most 2. We are able to show the following result for triangles:

\begin{theorem}
The number of copies of $K_3$, $X_n^{\blacktriangle}$, on a uniformly at random series-parallel graph with $n$ vertices
is asymptotically normal, with
$$
\mathbb{E}[X_n^{\blacktriangle}] = \mu_{\blacktriangle} n+O(1),\qquad \mathbb{V}\mathrm{ar}[X_n^{\blacktriangle}] = \sigma_{\blacktriangle}^2 n+O(1),
$$
where $\mu_{\blacktriangle} \approx 0.39481$ and $\sigma_{\blacktriangle}^{2} \approx 0.41450$.
\end{theorem}
Our encoding also let us analyze the asymptotic number of triangle-free series-parallel graphs on $n$ vertices. Also, the more involved case of studying the number of copies of $C_4$, as well as series-parallel graphs with a given girth is discussed in Section \ref{sec:comp}.
$$\,$$
\paragraph{\textbf{Plan of the paper}:} the paper is divided in the following way. Section 2 is devoted to fix the notation concerning generating functions. Section 3 covers the analytic preliminaries of the paper. This section includes a modified version of the main theorem of Drmota, Gittenberger and Morgenbesser in \cite{DrmGittMor2015}, which is our main analytic ingredient on the proof of Theorem \ref{thm:main}. Section 4 deals with the easier situation where the subgraph under study is 2-connected. The arguments to deal with the general connected case are developed in Section 5. In order to prepare the reader to the involved notation used to deal with general subgraphs, some easy cases are fully developed. Section 6 is devoted to explicit computations in the family of series-parallel graphs. Finally, Section \ref{sec:concluding} discusses the results obtained so far and possible future investigations.

%%%%%%%%%%%%%%%%%%%%%%%%%%%%%%%%%%%%%%%%%%%%%%%%%%%%%%%%%%PRELIMINARIES%%%%%%%%%%%%%%%%%%%%%%%%%%%%%%%%%%%%%%%%%%%%%%%%%%%%%%%%%%%%%%%%%%%%%%%%%%%%%
%%%%%%%%%%%%%%%%%%%%%%%%%%%%%%%%%%%%%%%%%%%%%%%%%%%%%%%%%%%%%%%%%%%%%%%%%%%%%%%%%%%%%%%%%%%%%%%%%%%%%%%%%%%%%%%%%%%%%%%%%%%%%%%%%%%%%%%%%%%%%%%%%%%%

\section{Graph preliminaries}\label{sec: gr-prel}

In our work, all graphs we study are assumed to be simple (no loops nor multiple edges) and labelled.
A graph on $n$ vertices will be always labelled with different elements in $\{1,\dots,n\}$.

\subsection{Combinatorial classes. Exponential generating functions.}\label{sec:def_series}

We follow the notation and definitions in \cite{fs05}.
A \emph{labelled combinatorial class} is a set $\mathcal{A}$ joint with a size function, such that for each $n\geq 0$ the set of elements of size $n$,
denoted by $\mathcal{A}_n$, is finite.
Each object $a$ in $\mathcal{A}_n$ is built by using $n$ labelled \emph{atoms} of size $1$. In graph classes, atoms are precisely the vertices.

Two elements in $\mathcal{A}$ are said to be \emph{isomorphic} if one is obtained from the other by relabelling.
In particular, two isomorphic elements have the same size.
We always assume that a combinatorial class is stable under relabelling, namely, $a\in \mathcal{A}$ if and only if all relabellings of $a$ are also elements of $\mathcal{A}$.
For counting reasons we consider the exponential generating function (shortly the EGF) associated to the labelled class $\mathcal{A}$:
$$
A(x):=\sum_{n\geq 0}|\mathcal{A}_n|{x^n\over n!}.
$$
In our setting, we use the (exponential) indeterminate $x$ to encode vertices.
In the opposite direction, we also write $[x^n]A(x)=1/n! |\mathcal{A}_n|$.
The basic constructions we consider in this paper are described in Table~\ref{fig:dict}.
In particular, we consider the disjoint union of labelled classes, the labelled product of classes, the sequence construction, the set construction, the cycle and the substitution (see \cite{fs05} for all the details).

We additionally consider classes of graphs of various types depending on whether one marks vertices or not.
A (vertex-)\emph{pointed graph} is a graph with a distinguished (labelled) vertex.
A \emph{derived graph} is a graph where one vertex is distinguished but not labelled (the other $n-1$ vertices have distinct labels in $\{1,\dots,n-1\}$).
In particular, isomorphisms between two pointed graphs (or between two derived graphs) have to respect the distinguished vertex.

Given a graph class $\mathcal{A}$, the \emph{pointed} class $\mathcal{A}^{\bullet}$ is the class of pointed graphs arising from $\mathcal{A}$.
Similarly, the \emph{derived} graph class $\mathcal{A}^{\circ}$ is obtained by taking all derived graphs built from $\mathcal{A}$.
Hence, $|\mathcal{A}^{\bullet}_n|=n|\mathcal{A}_n|$ and $|\mathcal{A}^{\circ}_{n-1}|=|\mathcal{A}_n|$, and we have respectively  $A^{\bullet}(x)=x\frac{d}{d x}A(x)$ and $A^{\circ}(x)=\frac{d}{d x}A(x)$.
\begin{table}[htb]
\begin{center}
\begin{tabular}{c|c|c}
Construction & Class & Equations \\
\hline
Sum             & $\mathcal{C}=\mathcal{A}\cup \mathcal{B}$     & $C(x)=A(x)+B(x)$\\
Product         & $\mathcal{C}=\mathcal{A}\times \mathcal{B}$   & $C(x)=A(x)\cdot B(x)$ \\
Sequence        & $\mathcal{C}=\mathrm{Seq}(\mathcal{B})$       & $C(x)=1/(1-B(x))$  \\
Set             & $\mathcal{C}=\mathrm{Set}(\mathcal{B})$       & $C(x)=\exp(B(x))$  \\
Restricted Set  & $\mathcal{C}=\mathrm{Set}_{\geq k}(\mathcal{B})$       & $C(x)=\exp_{\geq k} (B(x))=\exp(B(x))-\sum_{i=0}^{k-1}{B(x)^i\over i!}$  \\
Cycle           & $\mathcal{C}=\mathrm{Cyc}(\mathcal{B})$     & $C(x)=-\frac{1}{2}\log(1-x)-{x\over 2}-{x^2\over 4}$  \\
Substitution    & $\mathcal{C}=\mathcal{A}\circ\mathcal{B}$     & $C(x)=A(B(x))$ \\
Pointing        & $\mathcal{C}=\mathcal{A}^{\bullet}$           & $C(x)=A^{\bullet}(x)=x\frac{d}{d x}A(x)$ \\
Deriving        & $\mathcal{C}=\mathcal{A}^{\circ}$             & $C(x)=A^{\circ}(x)=\frac{d}{d z}A(x)$ \\
\end{tabular}
\end{center}
\caption{The Symbolic Method translating combinatorial constructions into operations on counting series.}
\label{fig:dict}
\end{table}

Pointing and deriving operators will be only used over vertices.
When dealing with \emph{ordinary} parameters over combinatorial classes (for instance, edges or copies of a fixed subgraph) we use extra variables in the corresponding counting formulas. The partial derivatives of counting series with respect to parameters are denoted by subindices of the corresponding indeterminate. For instance, a generating function of the form $A_y(x,y)$ means $\frac{d}{d y} A(x,y)$.

\subsection{Graph decompositions}\label{subsec:gr-decomp}

A \emph{block} of a graph $g$ is a maximal 2-connected subgraph of $g$.
A graph class $\mathcal{G}$ is \emph{block-stable} if it contains the edge-graph $e$ (the unique connected graph with two labelled vertices), and satisfies the property that a graph $g$ belongs to $\mathcal{G}$ if and only if all the blocks of $g$ belong to $\mathcal{G}$.
Block-stable classes covers a wide variety of natural graph families, including graph classes specified by a finite list of forbidden minors that are all 2-connected. Planar graphs ($\mathrm{Ex}(K_5,K_{3,3})$) or series-parallel graphs ($\mathrm{Ex}(K_4)$) are block-stable.

For a graph class $\mathcal{G}$, we write $\mathcal{C}$ and $\mathcal{B}$ the subfamily of connected and $2$-connected graphs in $\mathcal{G}$, respectively. In particular, the following combinatorial specifications hold (see~\cite{BerLaLe98,DrmotaBook,GoulJack83}):
\begin{equation*}
\mathcal{G}=\mathrm{Set}(\mathcal{C}),\,\,\, \mathcal{C}^{\bullet}=\bullet\times \mathrm{Set}(\mathcal{B}^{\circ}\circ \mathcal{C}^{\bullet} ).
\end{equation*}
By means of the Table \ref{fig:dict} these expressions translates into equations of EGF in the following way:
\begin{equation}\label{eq:main_eq}
G(x)=\exp(C(x)),\,\,\, C^{\bullet}(x)=x\exp(B^{\circ}(C^{\bullet}(x))).
\end{equation}
See \cite{tutte1966} for further results on graph decompositions and connectivity on graphs.

%%%%%%%%%%%%%%%%%%%%%%%%%%%%%%%%%%%%%%%%%%%%%%ANALYTIC PRELIM%%%%%%%%%%%%%%%%%%%%%%%%%%%%%%%%%%%%%%%%%%%%%%%%%%%%%%%%%%%%%%%%%%%%%%%%%%%%%%%%%%%%%%%
%%%%%%%%%%%%%%%%%%%%%%%%%%%%%%%%%%%%%%%%%%%%%%%%%%%%%%%%%%%%%%%%%%%%%%%%%%%%%%%%%%%%%%%%%%%%%%%%%%%%%%%%%%%%%%%%%%%%%%%%%%%%%%%%%%%%%%%%%%%%%%%%%%%%

\section{Analytic preliminaries}\label{sec: an-prel}

In this part we include the analytic results necessary in the forthcoming sections of the paper.

\subsection{Subcritical graphs}\label{sec:subcritical}
We start with the notion of subcritical graph class. Further details concerning these graph classes can be found in \cite{DrFuKaKrRu11}.

\begin{definition}
A block-stable class of (vertex labelled) graphs is called \emph{subcritical} if
\[
C^\bullet(\rho_C) < \rho_B,
\]
where $\rho_B$ denotes the radius of convergence of $B^\circ(x)$ and
$\rho_C$ the radius of convergence of $C^\bullet(x)$.
\end{definition}

Roughly speaking, subcritical condition means that the singular behaviour of $B^\circ(x)$ does not
interfer with the singular behaviour of $C^\bullet(x)$. Only the behaviour
of $B^\circ(x)= B'(x)$ for $|x| \le (1+\varepsilon)C^\bullet(\rho_C)$ matters
(where $\varepsilon>0$ is arbitrarily small).
From general theory (see for instance \cite{fs05}) it follows that $y = C^\bullet(x)$ becomes singular
for $x = x_0 = \rho_C$ if $x_0$ (and $y_0 = C^\bullet(x_0)$) satisfies the system
of equations
\begin{align*}
y_0 &= x_0 e^{B^\circ(y_0)},\\
1 &= x_0 e^{B^\circ(y_0)}{B^\circ}'(y_0),
\end{align*}
of equivalently if
\begin{align*}
1 &= y_0 B''(y_0),\\
x_0 &= y_0 e^{-B'(y_0)}.
\end{align*}
In particular, we just have to assure that the equation
$1 = y B''(y)$ has a solution $y < \rho_B$. Equivalently this is granted if
\[
\rho_B B''(\rho_B) > 1.
\]
It also follows from general theory that the solution function $C^\bullet(x)$ has a square-root type singularity at $x = \rho_C$ and can be (locally) written in the form
\[
C^\bullet(x) = h_1(x) - h_2(x) \sqrt{1- \frac x{\rho_C}},
\]
where $h_1(x)$ and $h_2(x)$ are analytic functions at $x = \rho_C$ and satisfy the condition $h_1(\rho_C) = C^\bullet(\rho_C)$ and $h_2(\rho_C) > 0$.

It is convenient to assume that our graph class is an \emph{$a$-periodic class}.
That is, we have $[x^n] C^\bullet(x) > 0$ for $n\ge n_0$.
Then it follows that $x=\rho_C$ is the only singularity on the circle of convergence $|x|=\rho_C$.
Additionally, there is an analytic continuation of $C^\bullet(x)$ to a domain of the form  $\{x \in \mathbb{C}: |x| < \rho' \,\, \mathrm{and}\,\, \arg(x - \rho_C)\not\in  [-\theta, \theta]\}$ for some real number $\rho' > \rho_C$ and some positive angle $0 < \theta < \pi/2$.
We call such a domain \emph{$\Delta$-region} or \emph{domain dented} at $\rho_C$.

More precisely, if $|x|= \rho_C$ but $x\ne \rho_C$ then
\[
|C^\bullet(x) {B^\circ}'(C^\bullet(x),1)| < 1.
\]
Thus by the Implicit Function Theorem $C^\bullet(x)$ has no singularity there and can be analytically continued.
Consequently, we get by singularity analysis over $C^\bullet(x) $ that
\[
[x^n] C^\bullet(x) = \frac{h_2(\rho_C)}{2\sqrt{\pi}} n^{-3/2} \rho_C^{-n}
\left( 1 + O\left( n^{-1} \right) \right).
\]
Since $C^\bullet(x) = x C'(x)$ we also obtain the local singular behavior of $C(x)$ which is of the form
\[
C(x) = h_3(x) + h_4(x) \left(1- \frac x{\rho_C}\right)^{3/2},
\]
for some functions $h_3(x)$ and $h_4(x)$ which are analytic at $x = \rho_C$. Since $G(x) = \exp(C(x))$ this also provides the local singular behavior
of $G(x)$:
\[
G(x) = h_5(x) + h_6(x) \left(1- \frac x{\rho_C}\right)^{3/2},
\]
where again $h_5(x)$ and $h_6(x)$ are analytic at $x = \rho_C$.
This implies (applying again singularity analysis) that
\[
[x^n] G(x) = \frac{3h_6(\rho_C)}{4\sqrt{\pi}} n^{-5/2} \rho_C^{-n}
\left( 1 + O\left( n^{-1} \right) \right).
\]
In what follows we will heavily make use of these properties of
subcritical graph classes.

\subsection{A single equation}

We first state a central limit theorem that is a slight modification of \cite[Theorem 2.23]{DrmotaBook}.
Let $F(x,y,{ u})= \sum_{n,m} F_{n,m}({ u}) x^n y^m$ be an analytic function in $x$, $y$ around $0$, and ${ u}$ is a complex parameter with $|u|=1$.
Suppose that the following conditions hold: $F(0,y,{ u}) \equiv 0$, $F(x,0,{ u})\not\equiv 0$ and all coefficients $F_{n,m}(1)$ of $F(x,y,1)$ are real and non-negative.
Suppose also that for $|u| = 1$  it is true that $|F_{n,m}(u)| \le F_{n,m}(1)$.
Finally, assume that the function $t\mapsto F(x,y,e^{it})$ is at least three times continuously differentiable and all derivatives are analytic, too, in $x$ and $y$. Then, by the implicit function Theorem it is clear that the functional equation
\begin{equation}
y = F(x,y,{ u})
\label{eq21.2}
\end{equation}
has a unique analytic solution $y=y(x,{ u}) = \sum_{n} y_n({ u})x^n$
with $y(0,{ u}) =0$ that is three times continuously differentiable with
respect to $t$ if $u=e^{it}$. Furthermore the coefficients $y_n({ 1})$ are non-negative.

It is easy to show that there exists an integer $d\ge 1$ and a residue class $r$ modulo $d$
such that $y_n(1) = 0$ if $n \not\equiv r (\bmod d)$. In order to simplify the following
presentation we assume that $d=1$ (namely, we discuss the $a$-periodic case).
The general case can be reduced to this case by a proper substitution in the original
equation.

We also assume that the region of convergence of
$F(x,y,{ u})$ is large enough such that
there exist non-negative solutions $x=x_0$ and $y=y_0$ of the
system of equations
\begin{align}\label{eq:singsyst}
y &= F(x,y,{ 1}),\\
1 &= F_y(x,y,{ 1}),\nonumber
\end{align}
with $F_x(x_0,y_0,1)\ne 0$ and $F_{yy}(x_0,y_0,{ 1})\ne 0$.

\begin{theorem}\label{Thanalytic1}
Let $F(x,y,u)$ satisfies the above assumptions and $y(x,u)$ is a power series in $x$ that is the (analytic) solution of
the functional equation $y = F(x,y,u)$. Suppose that $X_n$ is a sequence of random variables such that
\[
\mathbb{E}\left[u^{X_n}\right] = \frac{[x^n]\,y(x,u)}{[x^n]\, y(x,1)},
\]
where $|u|=1$. Set
\begin{align*}
\mu \,=& \,\frac {F_u}{x_0F_x},\\
\sigma^2\, =&\,\frac {1}{x_0F_x^3F_{yy}}\Bigl( F_x^2(F_{yy}F_{uu} - F_{yu}^2)
-2F_xF_u(F_{yy}F_{xu}-F_{yx}F_{yu}) + F_u^2(F_{yy}F_{xx} - F_{yx}^2) \Bigr)\\
&+\mu + \mu^2,
\end{align*}
where all partial derivatives are evaluated at the point $(x_0,y_0,1)$ solution to the system of equations \eqref{eq:singsyst}.
Then we have that
\[
\mathbb{E} [X_n] = \mu n + O(1),\, \mathbb{V}{\rm ar} [X_n] = \sigma^2 n + O(1)
\]
and if $\sigma^2> 0$ then
\[
\frac{X_n - \mathbb{E} [X_n]}{\sqrt{\mathbb{V}{\rm ar} [X_n]}} \rightarrow N(0,1).
\]
\end{theorem}

\begin{proof}
The proof runs along the same lines as that of \cite[Theorem 2.23]{DrmotaBook}.
We just indicate the differences.

By the Implicit Function Theorem it follows that there
there exist functions $f({ u})$ and $g({ u})$
(for $|u-1| < \varepsilon$ and $|u| = 1$ for some
$\varepsilon>0$)  which are three times differentiable  with
respect to $t$ if $u=e^{it}$ that satisfies
\begin{align*}
g(u) &= F(f(u),g(u),u),\\
1 &= F_y(f(u),g(u),u)
\end{align*}
with  $f(1) =x_0$ and $g(1) = y_0$.
Furthermore, by applying a proper variant of the Weierstrass Representation Theorem
it follows (as in the proof of \cite[Theorem 2.23]{DrmotaBook}) that we have a presentation of the form
\begin{equation}\label{eqsqrtrep2.2}
y(x,{ u}) = h_1(x,{ u}) - h_2(x,{ u}) \sqrt{1-\frac x{f({ u})}}
\end{equation}
locally around $x=x_0$, ${ u} = { 1}$, where $h_1(x,u)$, and $h_2(x,u)$ are analytic in $x$ and three times continuously differentiable with respect to $t$ if $u = e^{it}$, where $h_1(f(u),u) = g(u)$ and
\[
h_2(f(u),u) = \sqrt{\frac{2f({ u}) F_x(f({ u}),g(u),{ u})}
{F_{yy}(f({ u}),g(u),{ u})}} \ne 0.
\]
Since $d=1$ we also get
\begin{equation}\label{eqTh2A1last.2}
y_n({ u}) = [x^n] y(x,u) = \sqrt{\frac{f({ u}) F_x(f({ u}),g(u),{ u})}
{2\pi F_{yy}(f({ u}),g(u),{ u})}} f({ u})^{-n} n^{-3/2} \left( 1 + O( n^{-1})\right)
\end{equation}
uniformly for $|u-1|<\varepsilon$ and $|u| = 1$. Hence,
\begin{equation}\label{equXnbeh}
\mathbb{E}\left[u^{X_n}\right] =  \frac{[x^n] y(x,u)}{[x^n] y(x,1)}
= \frac{h_2(f(u),u)}{h_2(f(1),1) } \left(  \frac {f(1)} {f(u)} \right)^n
\left( 1 + O\left( n^{-1} \right) \right).
\end{equation}
By using the local expansion of $f(u)$ we get for $u = e^{it}$
\[
 \frac {f(1)} {f(u)} = e^{it\mu - \sigma^2 t^2/2 + O(t^3)},
\]
which directly implies
\[
\lim_{n\to\infty} \mathbb{E} \left[e^{it (X_n - \mu n)/(\sigma \sqrt n) } \right]
= e^{-t^2/2}.
\]
By Levi's Theorem this proves the central limit theorem.
\end{proof}

\begin{remark}\label{mainremark}
In our applications, the function $y(x,u)$ will be the generating function
$C^\bullet(x,u) = x \frac{\partial}{\partial x}C(x,u)$ of connected graphs.
Since $[x^n] C^\bullet(x,u) = n [x^n] C(x,u)$ it follows that
\[
\frac{[x^n]\,C(x,u)}{[x^n]\, C(x,1)} = \frac{[x^n]\,C^\bullet(x,u)}{[x^n]\, C^\bullet(x,1)}
\]
and, thus, it is sufficient to work with $C^\bullet(x,u)$ instead of $C(x,u)$.
However, if we are interested in all graphs (not necessarily connected) we need to study the behaviour of $G(x,y)$.
By means of the set construction $G(x,u) = \exp(C(x,u))$ we have to replace $y(x,u) = C^\bullet(x,u)$
by the function
\[
G(x,u) = \tilde y(x,u) = \exp\left( \int_0^x \frac{y(\xi,u)}{\xi} \, d\xi \right)
\]
and the new random variable $\tilde X_n$ that is defined by
$\mathbb{E}\left[u^{\tilde X_n}\right] =  \frac{[x^n] \tilde y(x,u)}{[x^n] \tilde y(x,1)}$.
Indeed, $\tilde y(x,u)$ has a slightly different singular behaviour: from (\ref{eqsqrtrep2.2})
we obtain
\[
 \int_0^x \frac{y(\xi,u)}{\xi} \, d\xi =  h_3(x,{ u}) + h_4(x,{ u}) \left(1-\frac x{f({ u})}\right)^{3/2}
\]
and consequently
\[
\tilde y(x,u) = h_5(x,{ u}) + h_6(x,{ u}) \left(1-\frac x{f({ u})}\right)^{3/2}
\]
for proper function $h_3(x,u),\, h_4(x,u),\, h_5(x,u),\, h_6(x,u)$. However, from that expression
we obtain the same kind of asymptotic behavior as in (\ref{equXnbeh}) and a central limit
theorem for $\tilde X_n$ with the same asymptotic behaviour for mean and variance
as for $X_n$.
\end{remark}

\begin{remark}\label{rem:sigma}
In most of the applications, the condition $\sigma^2 >0$ is satisfied. As it is shown in~\cite[Lemma 4]{DrFuKaKrRu11}, if $y = F(x,y,u)= \sum_{n,m,k} a_{n,m,k} x^n y^m u^k$ satisfies some natural analytic conditions (see \cite{DrFuKaKrRu11}), and assuming that there are three integer vectors $(n_j,m_j,k_j)$, $j = 1,2,3$ with $m_j> 0$, $j=1,2,3$ with
\[
\left| \begin{array}{rrr}
n_1 & m_1-1 & k_1 \\ n_2 & m_2-1 & k_2 \\n_3 & m_3-1 & k_3
\end{array} \right| \ne 0
\]
and $a_{n_j,m_j,k_j} \ne 0$ for $j=1,2,3$, then $\sigma^2>0$.
\end{remark}

\begin{remark}\label{rem:system}
Finally we remark that Theorem~\ref{Thanalytic1} extends to a finite system of equations
$y_j = F_j(x,y_1,\ldots,y_K,u)$, $1\le j \le K$, provided that the system is strongly connected
(compare with \cite[Theorem 2.35]{DrmotaBook}). We will use this extension in Section~\ref{ss.4-cycles-sp}.
\end{remark}

\subsection{An infinite system of equations}
The main reference for this subsection is the work \cite{DrmGittMor2015}.
We start again with an equation of the form $y = F(x,y)$, where $F$ satisfies (almost)
the same assumptions as that of Theorem~\ref{Thanalytic1} (we just omit the conditions concerning $u$).
In particular this means that the solution $y = y(x)$ has a square-root type singularity at
$x_0$ and the coefficients $y_n = [x^n] y(x)$ have an asymptotic expansion of the form
(\ref{eqTh2A1last.2}), where $u=1$.

Next, for a parameter $u$ with $|u| = 1$, we suppose that there exist functions $y_j(x,u)$, $j= 1,2,\ldots$, such that
\begin{equation}\label{eq:partition}
y(x) = \sum_{j\ge 1} y_j(x,1)
\end{equation}
and that the functions ${\bf y} = (y_j(x,u))_{j\ge 1}=(y_j)_{j\ge 1}$ satisfy the (infinite) system of equations
\begin{equation}\label{eqinfsystem}
y_j = F_j(x,{\bf y}, u),\qquad j\ge 1,
\end{equation}
where $F_j$ has a power series expansion
$$F_j(x,{\bf y},u) = \sum_{n,m_1,m_2,\ldots} F_{j;n,m_1,m_2,\ldots}(u) x^n y_1^{m_1}y_2^{m_2}\cdots$$
with coefficients that satisfy  $|F_{j;n,m_1,m_2,\ldots}(u)| \le F_{j;n,m_1,m_2,\ldots}(1)$.
In particular, these coefficients are non-negative for $u=1$.
Moreover, we assume that for every $j$ there exists a function $\tilde F_j(x,y)$ with
\begin{equation}\label{eqinfsystemcond1}
F_j(x,{\bf y}, 1) = \tilde F_j(x,y_1+y_2+\cdots)
\end{equation}
and that
\begin{equation}\label{eqinfsystemcond2}
\sum_{j\ge 1} \tilde F_j(x,y) = F(x,y).
\end{equation}
Informally speaking, this means that the infinite system can be interpreted as a partition of
the main equation $y = F(x,y)$. Hence, we refer to this later as the \emph{partition} Property.

From these properties it immediately follows that $F_j$ is well defined (and also analytic) for $x$ and
${\bf y} = (y_j)_{j\ge 1}$ for which $F(|x|, \|{\bf y}\|_1)$ is analytic (recall that $\|{\bf y}\|_1=\sum_{j\geq 1} |y_j|$).
Consequently, under the same conditions, $\tilde F_j(|x|, \|{\bf y}\|_1)$ is convergent. Actually we only need convergence for $|x|< x_0+\varepsilon$ and $\|{\bf y}\|_1 < y_0 + \varepsilon$ for some $\varepsilon> 0$.

This property suggests to work in the space $\ell^1(\mathbb{C})$ for ${\bf y} = (y_j)_{j\ge 1}$.
However, in the present situation we have to be slightly more careful since we have to take also into account derivatives with respect to
$t$ (with $u = e^{it}$).
For this purpose we use weighted $\ell^1$ spaces of the form
\[
\ell^1(m,\mathbb{C}) = \Big\{ {\bf y} = (y_j)_{j\ge 1} \in \mathbb{C}^{\mathbb{N}} :
\|{\bf y }\|_{m,1} := \sum_{j\ge 1} j^m |y_j| < \infty \Big\},
\]
for some non-negative real number $m$ (see also Remark~\ref{mainremark2}).
Since $\|{\bf y }\|_{1} \le \|{\bf y }\|_{m,1}$ the functions $F_j$ are also well defined (and analytic) if
$|x|< x_0+\varepsilon$ and $\|{\bf y}\|_{m,1} < y_0 + \varepsilon$
for some $\varepsilon> 0$.

Finally we assume that, for each $j$, $F_j$ is three times continuously differentiable with
respect to $t$ with $u = e^{it}$ such that the series
\begin{equation}\label{eqderivatives}
\sum_{j\ge 1} j^m\frac{\partial^r }{\partial t^r} F_j(x,{\bf y}, e^{it}), \qquad r\in \{0,1,2,3\},
\end{equation}
converges absolutely for $|x|< x_0+\varepsilon$ and $\|{\bf y}\| < y_0 + \varepsilon$
(for some $\varepsilon> 0$). Note that the case $r=0$ just says that for each $j$ the mapping
$(x,{\bf y}) \mapsto F_j(x,{\bf y},u)$ is well defined in the space $\mathbb{C} \times \ell^1(m,\mathbb{C})$ with
$|x|< x_0+\varepsilon$ and $\|{\bf y}\|_{m,1} < y_0 + \varepsilon$ (for some $\varepsilon> 0$).

The main theorem in this context is the following:

\begin{theorem}\label{Thanalytic2}
Let $y(x,u)$ be a power series in $x$, $y(x,u) = \sum_{j\ge 1} y_j(x,u)$, where the set of power series $y_j= y_j(x,u)$, $j\ge 1$, satisfies an
infinite system of equations $y_j = F_j(x,{\bf y},u)$ satisfying the above assumptions. Suppose that $X_n$ is a sequence of random variables such that
\[
\mathbb{E}\left[ u^{X_n}\right] = \frac{[x^n]\,y(x,u)}{[x^n]\, y(x,1)}
\]
for $|u|=1$. Then we have
\[
\mathbb{E}[X_n] = \mu n + O(1) \quad\mbox{and}\quad
\mathbb{V}{\rm ar}[X_n] = \sigma^2 n + O(1)
\]
for some real constants $\mu> 0$ and $\sigma^2\ge 0$. Furthermore if $\sigma^2> 0$ then
\[
\frac{X_n - \mathbb{E}[X_n]}{\sqrt{\mathbb{V}{\rm ar}[X_n]}} \to N(0,1).
\]
\end{theorem}

\begin{remark}\label{mainremark2}
We note that a corresponding theorem for a finite system is also true (\cite{drmota97systems, DrmotaBook})
but in our context we just need the infinite version.

Furthermore, Theorem~\ref{Thanalytic2} even holds in slightly more general situations.
For example, if the functions $y_j(x,u)$ are not indexed by an integer $j\ge 1$ but by a multi-index ${\bf j} = (j_1,\ldots,j_d)$ of integers $j_i\ge 1$ then we can also adapt the space $\ell^1(m,\mathbb{C})$ to the space
\[
\Big\{ {\bf y} = (y_{\bf j})_{{\bf j}\ge {\bf 1}} \in \mathbb{C}^{\mathbb{N}^d} :
\|{\bf y }\|_{m,1} := \sum_{{\bf j}\ge {\bf 1}} \|{\bf j}\|_1^m |y_{\bf j}| < \infty \Big\}.
\]
Actually we will need this generalization if we consider subgraphs $H$ with more than one cut-vertex.

Finally, as for Theorem~\ref{Thanalytic1}  the central limit theorem transfers to $\tilde X_n$ that is defined with the help of
$\tilde y(x,u) = \exp\left( \int_0^x y(\xi,u)/\xi\, d\xi \right)$. Compare this fact with Remark~\ref{mainremark}.
\end{remark}

\begin{proof}
We first note that Theorem~\ref{Thanalytic2} will be deduced from \cite[Theorem 1]{DrmGittMor2015} with a slight adaption corresponding $u$ -- here we just
require differentiability with respect to $t$ if $u = e^{it}$ and not analyticity --
and corresponding to the underlying space -- we replace $\ell^1(\mathbb{C})$ by
$\ell^1(m,\mathbb{C})$.
Actually the modification corresponding to $u$ can be treated as in the proof of Theorem~\ref{Thanalytic1} and the change of the
underlying space does not change the proof at all, so we will not discuss these issues.

Next we note that (\ref{eqinfsystemcond1}) implies
\[
y_j(x,1) = \tilde F_j(x, y(x)),
\]
where $y=y(x)$ is the solution of the equation $y = F(x,y)$.
Thus, we study two cases.
First, if $\tilde F_j$ does not depend on $y$ then $y_j(x,1)$ is analytic at $x=x_0$.
This also implies that $y_j(x,u)$ is analytic for $|u| = 1$ and for $|x| < 1 + \varepsilon$ for some $\varepsilon> 0$.
Let $I_1$ denote the set of indices $j$ with this property.
Furthermore, since $F(x,y)$ is also analytic at $x = x_0$ it also follows that $\sum_{j\in I_1} y_j(x,u)$ is analytic in $x$
for $|x| < x_0 + \varepsilon$ and for $|u| = 1$.

In the second case $y_j(x,1)$ has a square-root singularity of the form
$$y_j(x) = h_{1,j}(x) - h_{2,j}(x)\sqrt{1- x/x_0},$$
which is inherited from that of $y(x)$.
Furthermore it follows that $F_j(x,y_1,y_2,\ldots,u)$ depends on all variables $y_i$, $i\ge 1$. Let $I_2$ denote the set of indices $j$ of the second case.

If we reduce now the infinite system to those equations with $j\in I_2$, where we consider $y_j(x,u)$
with $j\in I_1$ as already known functions, then we get a strongly connected
system of equations
\[
y_j = F_j(x,(y_i(x,u))_{i\in I_1}, (y_\ell)_{\ell \in I_2}, u),\qquad j\in I_2
\]
that satisfies all regularity assumptions of \cite[Theorem 1]{DrmGittMor2015}.
In particular, since
\[
|F_j(x,y_1,y_2,\ldots, u)| \le F_j(|x|,|y_1|,|y_2|,\ldots,1) = \tilde F_j(|x|, |y_1|+|y_2|+\cdots)
\]
and $\tilde F_j(x,y)$ is analytic (at least) in the region where $F(x,y)$ is analytic, it
follows that the function $F_j(x,y_1,y_2,\ldots, u)$ is well defined (and analytic in $x$ and $y_1,y_2,\ldots$)
for $x$ in a proper neighborhood of $0$, ${\bf y} = (y_j)_{j\ge 1}$ in a proper neighborhood
of $0$ in $\ell^1(m,\mathbb{C})$ and $u$ with $|u|=1$.

The only remaining assumption that has to be checked is that the operator
\[
J =\left( \frac{\partial F_j}{\partial y_i}(x, {\bf y}, 1) \right)_{i,j\in I_2}
\]
is compact.
Since the property $$F_j(x,y_1,y_2,\ldots, 1) = \tilde F_j(x,y_1+y_2+\cdots)$$ is satisfied, it follows
that $$\frac{\partial F_j}{\partial y_i}(x, {\bf y}, 1) =
\frac{\partial \tilde F_j}{\partial y}(x, y_1+y_2+\cdots)$$ is independent of the choice of $i$.
Hence the rank of $J$ equals $1$ which implies that $J$ is a compact operator.

Thus we can apply \cite[Theorem 1]{DrmGittMor2015} and obtain that all functions $y_j(x,u)$, $j\in I_2$,
have a common square-root type singularity, and an expression of the form
\[
y_j(x,{ u}) = h_{1,j}(x,{ u}) - h_{2,j}(x,{ u}) \sqrt{1-\frac x{f({ u})}}.
\]
with functions $f(u), h_{1,j}(x,u), h_{2,j}(x,u)$ that are three times
differentiable in $t$, where $u= e^{it}$ and analytic in $x$ around $x_0$.

Summing up we, thus, obtain a square-root singularity for $y(x,u)$.
So we are precisely in the same situation as in the proof of Theorem~\ref{Thanalytic1}.
And so the result follows.	
\end{proof}

%%%%%%%%%%%%%%%%%%%%%%%%%%%%%%%%%%%%%%%%%%%%%%%%%%%%%%%%%%%%%%%%%%%%%%%%%%%%%%%%%%%%%%%%%%%%%%%%%%%%%%%%%%%%%%%%%%%%%%%%%%%%%%%%%%%%%%%%%%%%%%%%%%%%
%%%%%%%%%%%%%%%%%%%%%%%%%%%%%%%%%%%%%%%%%%%%%%%%%%%%%%%%%%%%%%%%%%%%%%%%%%%%%%%%%%%%%%%%%%%%%%%%%%%%%%%%%%%%%%%%%%%%%%%%%%%%%%%%%%%%%%%%%%%%%%%%%%%%

\section{$2$-Connected Subgraphs}\label{sec:2-con}

The purpose of this section is to consider 2-connected subgraphs $H$.
This case is much easier than the general case since a 2-connected
subgraph can only appear in a block.
Due to its shortness, we include the proof for this specific subgraph case.

\begin{theorem}\label{Th2.1}
Suppose that $H$ is a $2$-connected graph that appears as a subgraph in a(n a-periodic)
subcritical graph class $\mathcal{G}$. Let $X_n = X_n^{(H)}$ denote
the number of occurences of $H$ as a subgraph in a connected or general random graph
of size $n$ of $\mathcal{G}$.

Then, $X_n$ satisfies a central limit theorem with
$\mathbb{E}\, X_n \sim \mu n$ and
$\mathbb{V}{\rm ar}\, X_n \sim \sigma^2 n$ as $n\to\infty$.
\end{theorem}

\begin{proof}
Let $b_{n,k}^\circ$ the number of rooted $2$-connected graphs in $\mathcal{G}$
with $n-1$ non-root vertices such that $H$ appears precisely $k$ times as a subgraph.
Furthermore let
\[
B^\circ(x,u) = \sum_{n,k} b_{n,k}^\circ \frac {x^n}{n!} u^k
\]
be the corresponding generating function.

Let $C^\bullet(x,u)$ be the corresponding generating function of connected graphs
in $\mathcal{G}$ (where the root is non discounted).
Since $H$ is assumed to be $2$-connected the number of occurrences
of $H$ in a connected graph is just the sum of its occurrences in the $2$-connected
components. Hence we have
\[
C^\bullet(x,u) = x e^{B^\circ (C^\bullet(x,u),u)}.
\]
If $u=1$ then $B^\circ(x,1)$ and $C^\bullet(x,1)$ are the usual counting functions
that satisfy the equation  $C^\bullet(x,1) = x e^{B^\circ (C^\bullet(x,1),1)}$.

In order to prove Theorem~\ref{Th2.1} we just have to check the conditions
of Theorem~\ref{Thanalytic1}. By the subcritical condition we certainly have $x=x_0 = \rho_C$ and
$y=y_0 = C^\bullet(\rho_C,1)$ that satisfy
\begin{align*}
1 &= x_0 e^{B^\circ(y_0,1)}{B^\circ}'(y_0,1),\\
 y_0 &= x_0 e^{B^\circ(y_0,1)}.
\end{align*}
Furthermore, since $C^\bullet(\rho_C,1) < \rho_B$ the region of convergence
of $F(x,y,u) = x e^{B^\circ (y,u)}$ is large enough.

The only missing assumption that has to be (finally) checked is that
the mapping $t\mapsto x e^{B^\circ (y,e^{it})}$ is three times continuously
differentiable in $t$. Of course it is sufficient to study the mapping
$t\mapsto B^\circ (y,e^{it})$.
First  we note that
$|B^\circ(y,u)| \le B^\circ(|y|,1)$. From this it follows that
${B^\circ}(y,u)$ exists (and is also analytic in $y$) for all
$|y| < \rho_B$ and for $|u|=1$.
Next we note that the number of occurrences of
a graph $H$ of size $L$ in a graph with $n$ vertices is bounded by $O(n^L)$.
Write $b_n^\circ$ for the number of rooted 2-connected graphs in $\mathcal{G}$ with $n-1$ non-root vertices. Thus is follows that
\[
\left| \frac{\partial^r}{\partial u^r} \sum_k  b_{n,k}^\circ u^k \right|
\le  \sum_k k^r b_{n,k}^\circ = O( n^{rL} b_n^\circ )
\]
for $u$ with $|u| = 1$; for notational convenience we have taken the derivatives
formally with respect to $u$. However, since all all derivatives
$\frac{\partial^m}{\partial y^m} B^\circ(y,1)$ are finite it follows that
all derivatives $\frac{\partial^r}{\partial u^r} B^\circ(y,u)$ exist for $|u| = 1$.
(Alternatively we can use the bound $n^{rL} = O((1+\varepsilon)^n)$ for every
$\varepsilon> 0$ which implies that
$$\left| \frac{\partial^r}{\partial u^r} B^\circ(y,u) \right| = O(B^\circ(|y|(1+\varepsilon),1)).$$
Consequently all assumptions of Theorem~\ref{Thanalytic1} are satisfied and the
result follows for the connected case. In the general case, where we have to
work with $G(x,u) = \exp\left( \int_0^x C^\bullet(\xi,u)/\xi\, d\xi \right)$,
we get the same result, see Remark~\ref{mainremark}.
\end{proof}

%%%%%%%%%%%%%%%%%%%%%%%%%%%%%%%%%%%%%%%%%%%%%%%%%%%%%%%%%%%%%%%%%%%%%%%%%%%%%%%%%%%%%%%%%%%%%%%%%%%%%%%%%%%%%%%%%%%%%%%%%%%%%%%%%%%%%%%%%%%%%%%%%%%%
%%%%%%%%%%%%%%%%%%%%%%%%%%%%%%%%%%%%%%%%%%%%%%%%%%%%%%%%%%%%%%%%%%%%%%%%%%%%%%%%%%%%%%%%%%%%%%%%%%%%%%%%%%%%%%%%%%%%%%%%%%%%%%%%%%%%%%%%%%%%%%%%%%%%

\section{Connected Subgraphs}\label{sec:conn}

The purpose of this section is to extend Theorem~\ref{Th2.1} to subgraphs
$H$ that are not 2-connected, and hence prove the main Theorem \eqref{thm:main}.
%\begin{theorem}\label{Th3.1}
%Suppose that $H$ is a connected graph that appears as a subgraph in a (a-periodic)
%subcritical graph class $\mathcal{G}$. Let $X_n = X_n^{(H)}$ denote
%the number of occurrences of $H$ as a subgraph in a connected or general random graph
%of size $n$ of $\mathcal{G}$.
%
%The $X_n$ satisfies a central limit theorem with $\mathbb{E} [X_n] \sim \mu n$
%and $\mathbb{V}{\rm ar} [X_n] \sim \sigma^2 n$ as $n\to\infty$.
%\end{theorem}
The main difference between the 2-connected case and the (general) connected case is that occurrences of $H$ are not necessarily separated by cut-vertices. This means that we have to cut $H$ also into pieces (more precisely, into blocks) and to count all combinations of these pieces when two (or several) blocks are joint by a cut-vertex, or several cut-vertices.

We start this section by illustrating the arguments with the base case $H=P_2$, which is the simplest case of a graph $H$ that is not 2-connected.
Later, as a warm-up for the general case (where notation could be specially involved), we show the combinatorics behind two particular cases: copies of subgraphs with 1 cut-vertex and exactly 3 blocks (Subsection \ref{subs.1-cut}) and the number of copies of $P_3$ (Subsection \ref{subs.P3}). In both cases we show again the type of functional equations we obtain in this setting and the main difficulties that arise when encoding the counting formulas.
At the end of the section we indicate how the method can be modified to cover the general case, both combinatorially and analitically.

\subsection{Counting copies of $P_2$}\label{subs.P2}
Despite this example does not cover the full general case, it is important to say that in the proof of the main theorem in Subsection \ref{subs.general} we will a use similar type of arguments one we find a convenient encoding.

If $H$ is just a path $P_2$ of length 2 then the situation is relatively simple since $P_2$ just separates by a cut-vertex into two edges.
For example if we join two blocks at a cut-vertex and the two corresponding degrees of two blocks at this cut-vertex are $k_1$ and $k_2$ then there are $k_1k_2$ occurrences of $P_2$ just coming from this connection.
Actually it turns out that we have to distinguish between infinitely many situations (depending on the root degrees) which leads to an infinity system
of equations. Let us start by introducing corresponding generating functions for 2-connected graphs.
We denote by
\[
B_j^\circ(w_1,w_2,w_3,\ldots;u), \qquad j\ge 1,
\]
the generating function of blocks in $\mathcal{G}$, where
the root has degree $j$, where $w_i$ counts the number of non-root vertices
of degree $i$, and where $u$ counts the number of occurrences of $H = P_2$.
Formally this is a generating function in infinitely many variables.
Of course we have
\[
B_j^\circ(x,x,\ldots;u) = B_j^\circ(x,u),
\]
where $x$ counts the number of non-root vertices.
Consequently if $B_j^\circ(x,u)$ is convergent for some positive $x$ and
for $u$ with $|u| = 1$ then $B_j^\circ(w_2,w_3,\ldots;u)$ converges
for all $w_i$ with $|w_i| < x$ and for all $u$ with $|u| = 1$. Next let
\[
C_j^\bullet(x,u), \qquad j\ge 0,
\]
denote the generating function of connected rooted graphs in $\mathcal{G}$, where
the root vertex has degree $j$, where $x$ counts the number of (all) vertices and
$u$ the number of occurrences of $H = P_2$. Then by the same principle as above
we have $C_0^\bullet(x) = x$ and for $j\ge 1$
\begin{equation}\label{eqCk}
C_j^\bullet(x,u) = x \sum_{s\ge 0} \frac 1{s!}
\sum_{j_1+\cdots+ j_s = j} u^{\sum_{i_1 < i_2} j_{i_1}j_{i_2}}
\prod_{i=1}^s B_{j_i}^\circ\left( \sum_{\ell_1\ge 0} u^{\ell_1} C_{\ell_1}^\bullet(x,u),
\sum_{\ell_2\ge 0} u^{2\ell_2} C_{\ell_2}^\bullet(x,u),\ldots; u \right)
\end{equation}
It is convenient to replace all occurrences of $C_0^\bullet(x)$ by $x$. Thus we can
see the infinite dimensional vector ${\bf y} = {\bf y}(x,u) = (C_j^\bullet(x,u))_{j\ge 1}$
as the solution of an infinite dimensional system
of the form $y_j = F_j(x,{\bf {y}}, u)$, where $F_j$ is defined by
\begin{equation}\label{eq:Fj}
F_j(x,{\bf y},u) =
x \sum_{s\ge 1} \frac 1{s!}
\sum_{j_1+\cdots+ j_s = j} u^{\sum_{i_i < i_2} j_{i_1}j_{i_2}}
\prod_{i=1}^s B_{j_i}^\circ\left( x + \sum_{\ell_1\ge 1} u^{\ell_1} y_{\ell_1},
x+ \sum_{\ell_2\ge 1} u^{2\ell_2} y_{\ell_2},\ldots; u \right).
\end{equation}
We now show that this system of equations satisfies all assumptions of Theorem~\ref{Thanalytic2}.
First of all, let us check the partition property.
If $u=1$ the function $F_j(x,{\bf{y}},1)$ can be written as a function $\tilde F_j(x,y_1+y_2+\cdots)$, where
\[
\tilde F_j(x,y) =
x \sum_{s\ge 1} \frac 1{s!}
\sum_{j_1+\cdots+ j_s = j} \prod_{i=1}^s
B_{j_i}^\circ\left( x + y, 1 \right).
\]
In particular, $F(x,y) = \sum_j \tilde F_j(x,y)$ is equal to $x \exp( B^\circ(x+y) )$, which is analytic in $x$.
Since $|\tilde F_j(x,y)| \le \tilde F_j(|x|,|y|)$ it is sufficient to study
$ \tilde F_j$ for positive $x$ and $y$.
By Equation \eqref{eq:Fj} it follows that for all $n < j$ we have $[x^n] \tilde F_j(x,y) = 0$.
Consequently we have (for positive $x$ and $y$)
\[
j^m \tilde F_j(x,y) \le \left( x \frac{\partial}{\partial x} \right)^m \tilde F_j(x,y).
\]
By analyticity of $\sum_j \tilde F_j(x,y)$, it directly follows then that
\[
\sum_{j\ge 1} j^m \tilde F_j(x,y) \le \left( x \frac{\partial}{\partial x} \right)^m F(x,y).
\]
Thus, the infinite system is well defined (and analytic) on $\ell^1(m,\mathbb{C})$ for
every $m \ge 0$.

It remains to check Condition \eqref{eqderivatives} for $r \in \{1,2,3\}$.
For the sake of brevity we only work out the details of the case $r=1$. The remaining cases are more involved but
can be handled similarly.
We first note that $u$ appears in $F_j$ at three different places:

\begin{enumerate}
\item[(1)] as the power $u^{\sum_{i_i < i_2} j_{i_1}j_{i_2}}$,
\item[(2)] in sums of the form $\sum_{\ell \ge 1} u^{m\ell} y_\ell$ as an argument of one of the terms $B_{j_i}^\circ$ and
\item[(3)] as the last argument in one of the terms $B_{j_i}^\circ(w_1,w_2,\ldots;u)$.
\end{enumerate}

As above it is sufficient to consider positive $x$ and $y = y_1+y_2+ \cdots$ in order
to assure absolute convergence.

If we substitute $u =e^{it}$ and take the derivative with respect to $t$ it follows that
in Case (1) the derivative gives a factor of the form
\[
i\sum_{i_i < i_2} j_{i_1}j_{i_2} e^{it\sum_{i_i < i_2} j_{i_1}j_{i_2}}
\]
which can be absolutely bounded by
\[
\sum_{i_i < i_2} j_{i_1}j_{i_2} \le s(s-1) j^2.
\]
Thus we are led to consider the sum (which is an upper bound)
\begin{align*}
\sum_{j\ge 1} j^m x \sum_{s\ge 1} \frac 1{s!} s(s-1) j^2
& \sum_{j_1+\cdots+ j_s = j} \prod_{i=1}^s
B_{j_i}^\circ\left( x + y, 1 \right) \\
& \leq  \sum_{j\ge 1} j^{m+2} x \sum_{s\ge 2} \frac 1{(s-2)!}
\sum_{j_1+\cdots+ j_s = j} \prod_{i=1}^s
B_{j_i}^\circ\left( x + y, 1 \right) \\
&\le  \left( x \frac{\partial}{\partial x} \right)^{m+2}
\sum_{j\ge 1} x \sum_{s\ge 2} \frac 1{(s-2)!}
\sum_{j_1+\cdots+ j_s = j} \prod_{i=1}^s
B_{j_i}^\circ\left( x + y, 1 \right) \\
& =  \left( x \frac{\partial}{\partial x} \right)^{m+2}
 x B^\circ(x+y)^2 \exp\left( B^\circ(x+y) \right)
\end{align*}
which is certainly bounded (for positive $x$ and $y$). Now we study Case (2).
If we take derivatives we get
\begin{align*}
&\frac{\partial}{\partial t}
B_{j}^\circ\left( x + \sum_{\ell_1\ge 1} e^{it\ell_1} y_{\ell_1},
x+ \sum_{\ell_2\ge 1} e^{2it\ell_2} y_{\ell_2},\ldots; u \right) \\
&\quad = \sum_{m\ge 1} \frac{\partial}{\partial w_m}  B_{j}^\circ\left( x + \sum_{\ell_1\ge 1} e^{it\ell_1} y_{\ell_1},
x+ \sum_{\ell_2\ge 1} e^{2it\ell_2} y_{\ell_2},\ldots; u \right)
\sum_{\ell_m \ge 1} im \ell_m e^{it\ell_m} y_{\ell_m}
\end{align*}
which can be bounded from the above by
\[
\left( \sum_{\ell \ge 1} \ell |y_\ell| \right) \sum_{m \ge 1} m  \frac{\partial B_{j}^\circ} {\partial w_m} (x+y,1)
%\le \left( \sum_{\ell \ge 1} \ell |y_\ell| \right) j  \frac{\partial B_{j}^\circ} {\partial y} (x+y,1).
\]
Note that the sum $\sum_{m \ge 1} m  \frac{\partial B_{j}^\circ} {\partial w_m}$ corresponds to the sum
of the degrees of the non-root vertices. Since this sum is bounded by twice the number of edges it is bounded by
$n(n-1)$, where $n$ denotes the number of vertices. This leads us to the upper bound
\[
\left( \sum_{\ell \ge 1} \ell |y_\ell| \right) (x+y)^2  \frac{\partial^2 B_{j}^\circ} {\partial x^2} (x+y,1).
\]
This upper bound also implies the upper bound (recall that the derivative here is only restricted to Case (2)):
\[
\left|\frac{\partial}{\partial t} \prod_{i=1}^s  B_{j_i}^\circ \right| \le
\left( \sum_{\ell \ge 1} \ell |y_\ell| \right) (x+y)^2  \frac{\partial^2} {\partial x^2}
\prod_{i=1}^s  B_{j_i}^\circ (x+y,1).
\]
Finally, summing up over $j$ (with the weight $j^m$) we obtain the upper bound (for positive $x$ and $y$)
\begin{align*}
& \left( \sum_{\ell \ge 1} \ell |y_\ell| \right)  \sum_{j\ge 1} j^{m} x
 \sum_{s\ge 1} \frac 1{s!}
\sum_{j_1+\cdots+ j_s = j}  (x+y)  \frac{\partial^2} {\partial x^2}
\prod_{i=1}^s  B_{j_i}^\circ (x+y,1) \\
&\quad \le  \left( \sum_{\ell \ge 1} \ell |y_\ell| \right)
\left( x \frac{\partial}{\partial x} \right)^{m+2} (x+y)^3  \exp\left( B^\circ(x+y) \right).
\end{align*}

%\begin{align*}
%& \left( \sum_{\ell \ge 1} \ell |y_\ell| \right)  \frac{\partial} {\partial y}\sum_{j\ge 1} j^{m+1} x(x+y)
%\frac{\partial^2} {\partial x^2} \sum_{s\ge 1} \frac 1{s!}
%\sum_{j_1+\cdots+ j_s = j}  (x+y)  \frac{\partial^2} {\partial x^2}
%\prod_{i=1}^s  B_{j_i}^\circ (x+y,1) \\
%&\quad \le  \left( \sum_{\ell \ge 1} \ell |y_\ell| \right)
%\left( x \frac{\partial}{\partial x} \right)^{m} x(x+y) \frac{\partial^2} {\partial x^2} \exp\left( B^\circ(x+y) \right).
%\end{align*}

By assumption we know that $\sum_{\ell \ge 1} \ell |y_\ell|$ is bounded.
Hence, the whole term is bounded. Finally in Case (3) we can argue in the same way as in the proof of
Theorem~\ref{Th2.1} and obtain
\[
\left|\frac{\partial}{\partial t} B_j^\circ(x+y,e^{it}) \right| \le j^3 B_j^\circ(x+y,1)
\]
(in the case of $H = P_2$ we have $L=3$). This leads us to consider the sum
\[
\sum_{j\ge 1} j^m x \sum_{s\ge 1} \frac 1{s!}
 \sum_{j_1+\cdots+ j_s = j} \left(\sum_{i=1}^s j_i^3 \right)\prod_{i=1}^s
B_{j_i}^\circ\left( x + y, 1 \right)
\]
which can be bounded (similarly to Case (1)) by
\[
\left( x \frac{\partial}{\partial x} \right)^{m+3}
x B^\circ(x+y)^3 \exp\left( B^\circ(x+y) \right).
\]

By putting the Case (1)--(3) together it follows that (\ref{eqderivatives}) is
satisfied for $r=1$. As mentioned above the cases $r=2$ and $r=3$ can be
similarly handled. This completes the proof of the central limit theorem in the case of $H = P_2$ for connected graphs.
%In the general case we can adapt Theorem~\ref{Thanalytic2} as mentioned in Remark~\ref{mainremark2}.

%\textbf{\textcolor{red}{Possibly the analytic argument can just be skipped as later we deal with it in the general case}}

%%%%%%%%%%%%%%%%%%%%%%%%%%%%%%%%%%%%%%%%%%%%%%%%%%%%%%%%%%%%%%%%%%%%%%%%%%%%%%%%%%%%%%%%%%%%%%%%%%%%%%%%%%%%%%%%%%%%%%%%%%%%%%%%%%%%%%%%%%%%%%%%%%%%
%%%%%%%%%%%%%%%%%%%%%%%%%%%%%%%%%%%%%%%%%%%%%%%%%%%%%%%%%%%%%%%%%%%%%%%%%%%%%%%%%%%%%%%%%%%%%%%%%%%%%%%%%%%%%%%%%%%%%%%%%%%%%%%%%%%%%%%%%%%%%%%%%%%%
\subsection{Main example 1: Connected graphs with 1 cut-vertex}\label{subs.1-cut}
In this subsection and in the following one we will motivate that the general statement for general subgraphs will be way more complicated than the analysis of the number of subgraphs $P_2$ carried out in the previous subsection.
We discusse next the number of copies of a connected graph $H$ with exactly 1 cut-vertex and three \emph{different} blocks
attached to it.
Let $H_1$, $H_2$ and $H_3$ denote these blocks, and $v$ the cut-vertex of $H$.
Furthermore we denote by $H_{1,2}$ the graph spanned by the vertices of $H_1$ and $H_2$, and
similarly $H_{1,3}$ and $H_{2,3}$.
The unique cut-vertex in $H$ induces a vertex in each $H_i$, $H_{i,j}$ that we denote by $c(H_i)$
and $c(H_{i,j})$, respectively.
% for the derived graph $H_i$ where the cut-vertex in $H$ is distinguished.
%Define analogously $H_{ij}^{\circ}$.
%Similarly, we consider pointed families $H_{i}^{\bullet}$ and $H_{ij}^{\bullet}$, where the pointed vertex is the unique cut-vertex of the graph.
All indices in this subsection are vectors with six components, of the form $L=(l_1,l_2,l_3; l_{1,2},l_{1,3},l_{2,3})$.
As we will show, such an index will encode the number of copies of $H_i$ and $H_{i,j}$ incident with a certain vertex.

Let $\textbf{w}$ be the infinite vector with components $w_{K}$, with $K=(k_1,k_2,k_3; k_{1,2},k_{1,3},k_{2,3})$ being an index with 6 entries.
We denote by $B_{L}^\circ(\textbf{w};u)$, $L=(l_1,l_2,l_3;l_{1,2},l_{1,3}, l_{2,3})$ the generating function of derived blocks in $\mathcal{G}$,
where the root vertex is incident with $l_i$ copies of $H_i$ ($i\in\{1,2,3\}$) and $l_{i,j}$ copies of $H_{i,j}$ ($i\neq j$) at $c(H_i)$ and $c(H_{i,j})$, respectively.
We use the variable $w_{K}$ to encode the number of vertices which are incident with $k_i$ copies of $H_i$, and $k_{i,j}$ copies of the subgraphs $H_{i,j}$ at $c(H_i)$.
We also use the variable $u$ to count the number of copies of $H$.
We note that different copies of the same subgraph $H_i$ or $H_{i,j}$ could be overlapping.
From the previous definition, it is obvious that writing $w_{K}=x$ in $B_{L}^\circ(\textbf{w};u)$ for all $K$ we obtain the generating function $B_{L}^\circ(x,u)$ where now $x$ counts the total number of vertices. As in the analysis of $P_2$, if this generating function is convergent for some positive $x$ and for $u$ with $|u| = 1$ then $B_{L}^\circ(\textbf{w};u)$ converges for all $w_{K}$ with $|w_{K}| < x$ and for all $u$ with $|u| = 1$.

For a vector index $R=(r_1,r_2,r_3; r_{1,2},r_{1,3},r_{2,3})$, let $C_{R}^\bullet(x,u)$ be the generating function of vertex-rooted connected graphs in $\mathcal{G}$, where the root vertex is incident with $r_i$ copies of $H_i$ at $c(H_i)$ and similarly for the numbers $r_{ij}$ and the subgraphs $H_{ij}$, respectively, and where $u$ counts the number of occurrences of $H$. Each of these functions satisfies the following equation
\begin{equation}\label{eqCkk}
C_{R}^\bullet(x,u)= x \sum_{s \geq 0} \frac{1}{s!} \sum_{\{L_1,\dots,L_s\}}^{\ast}\, u^{\sum_{i\neq j \neq k} (l_{1}^i l_{2}^j l_{3}^k)+ \sum_{i \neq j} \left(l_{1,2}^i l_{3}^j+ l_{1,3}^i l_{2}^j+ l_{1,2}^i l_{3}^j+l_{2,3}^{i}l_{1}^j\right)} \prod_{i=1}^{s} B_{L_i}^{\circ}(\textbf{W},u)
\end{equation}
where the sum $\sum_{i\neq j \neq k}$ is taken over triplets with pairwise different indices, and the sum $\sum_{\{L_1,\dots,L_s\}}^{\ast}$ is taken over all sets of $s$ indices $L_i=(l_{1}^i,l_{2}^i,l_{3}^i;l_{1,2}^i,l_{1,3}^i, l_{2,3}^i)$, $i=1,\dots,s$ satisfying
\begin{eqnarray*}
&&\sum_{i=1}^{s} l_{1}^i=r_1,\, \sum_{i=1}^{s} l_{2}^i=r_2,\, \sum_{i=1}^{s} l_{3}^i=r_3,\\
&&\sum_{i=1}^{s} l_{1,2}^i+\sum_{i\neq j} l_{1}^i l_{2}^j=r_{1,2},\, \sum_{i=1}^{s} l_{1,3}^i+\sum_{i\neq j} l_{1}^i l_{3}^j=r_{1,3},\, \sum_{i=1}^{s} l_{2,3}^i+\sum_{i\neq j} l_{2}^i l_{3}^j=r_{2,3},
\end{eqnarray*}
and the infinite vector $\textbf{W}$ has components
\begin{equation}\label{eq:Case2}
W_{K}= \sum_{P} u^{k_{1}p_{2,3}+k_{2}p_{1,3}+k_{3}p_{1,2}+k_{1,2}p_{3}+k_{1,3}p_{2}+k_{2,3}p_{1}} C_{P}^\bullet(x,u).
\end{equation}

Formula \eqref{eqCkk} reads in the following way: a pointed connected graph in the family where the root vertex is incident with $r_i$ copies of $H_i$ at $c(H_i)$ (and similarly for the numbers $r_{ij}$ and the subgraphs $H_{ij}$, respectively) is obtained by pasting a set of blocks at the root vertex, and adding the extra copies of $H$ created, both arising from the root vertex and for the composition of the blocks with the recursive copies of connected rooted objects. This last term is encoded by means of the term in $u$ after the sum $\sum_{L_i}^{\ast}$.

Let us show the analogy with the study of $P_2$.
Copies of $H$ may arise from 3 different sources:

\begin{enumerate}
\item  New copies that are incident with the root vertex.
\item  New copies not incident with the root vertex, built by taking subgraphs of $H$ already existing in the blocks and completing them.
\item  Copies already existing in the blocks incident with the root vertex.
 \end{enumerate}
In particular, Case (2) corresponds to the term in $u$ in Equation \eqref{eq:Case2}.
The analysis of this system of equations is very similar to the study of the number of copies of $P_2$ and can be mimic without any difficulty.
The only technical point in the analysis is that we have to check several properties in the functional space introduced in Remark~\ref{mainremark2}.

Let us also mention that the very similar arguments (with more indices) apply for subgraphsh $H$ with exactly one cut-vertex  (even with more than three blocks and possible block repetitions).

\subsection{Main example 2: Counting copies of $P_3$}\label{subs.P3}
We present an additional warm-up example, where we show a new difficulty that arises for subgraphs with more than one cut-vertex.
As we will see, it is not enough to express the infinite system of equations in terms of 'indexed' block families counting formulas.
Indeed, for each block in the class (and for each set of blocks) we will need a very precise information of its internal structure. It will turn out that Theorem~\ref{Thanalytic2} does not directly apply. However, we will
show how this problem can be overcome.

For illustrative reasons of this phenomenon, we just study the number of copies of $P_3$ on the subcritical class graph where all 2-connected blocks are isomgorphic to $K_4$ minus an edge.
This family is indeed subcritical due to the fact that the generating function for blocks is analytic (see \cite{3-con}).
We denote by $C_{k,l}^{\bullet}(x,u)$ the generating function of (vertex) rooted connected graphs in the family where the root vertex has degree $k$ and is the starting point of $l$ paths of type $P_2$ (possibly intersecting). As usual, $u$ marks occurrences of $P_3$.

In our setting, we have $B^\circ(x)=x^3$. Observe that (up to the labellings of the vertices) $K_4$ minus an edge has two different ways to be rooted: either over a vertex of degree 2 or degree 3. We call the resulting derived objects  $b_1^{\circ}$ and $b_2^{\circ}$ with generating functions $b_1^{\circ}(x) = b_2^{\circ}(x) = \frac 12 x^3$, respectively.

Let us now describe the system of equations satisfied by $C_{k,l}^{\bullet}(x,u)$, or at least the form of the first equations for small indices.
It is obvious that $C_{0,0}^{\bullet}(x,u)=x$, that for every choice of $l\neq 0$, $C_{0,l}^{\bullet}(x,u)=0$.
Also, for every choice of $l\geq 0$ $C_{1,l}^{\bullet}(x,u)=0$.
Expressions for $C_{2,l}^{\bullet}(x,u)$ and $C_{3,l}^{\bullet}(x,u)$ become more involved: in both cases we may have a block of type $b_1^{\circ}$ (and $b_2^{\circ}$, respectively) incident with the root of the connected object. See Figure \ref{graphic1} for a general structure of both cases.

\begin{figure}[htb]
\begin{center}
\includegraphics[width=3.5cm]{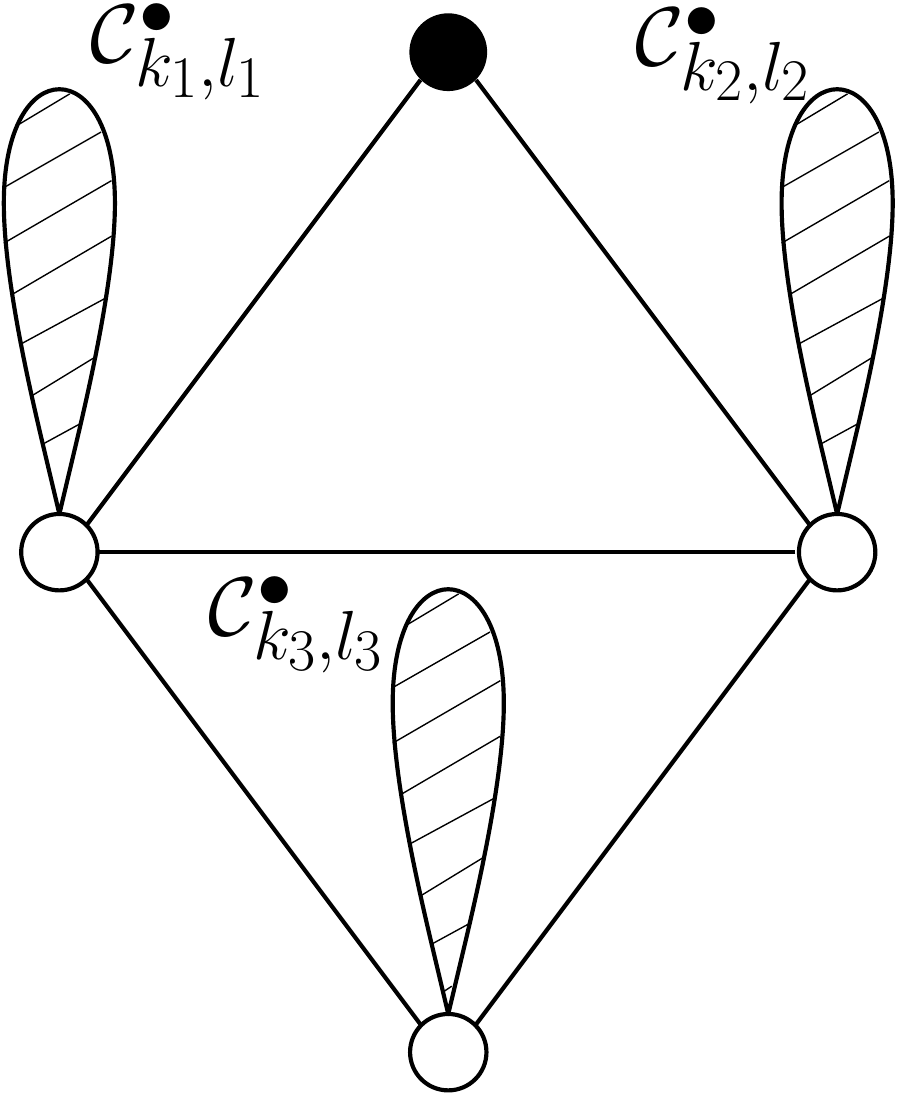}$\,\,\,\,\,\,\,\,$\includegraphics[width=3.5cm]{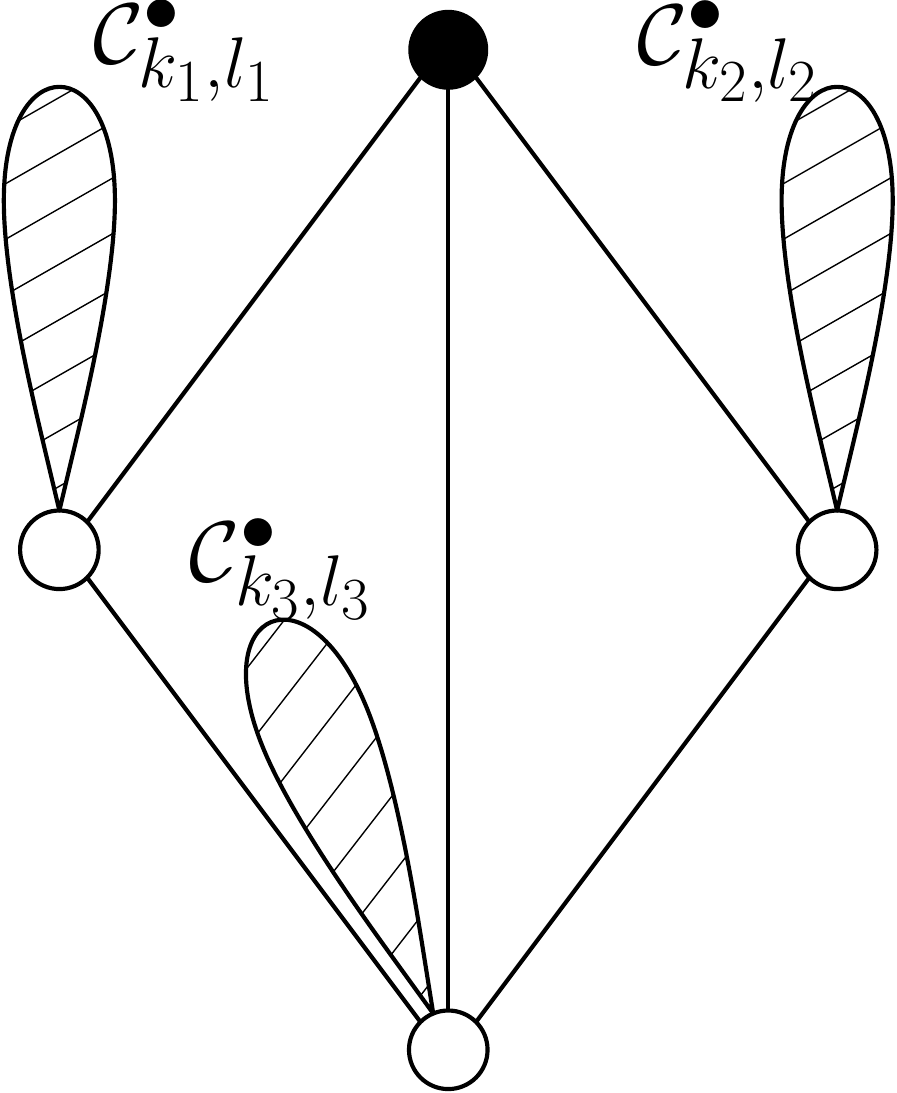}
\end{center}
\caption{ Block structure when dealing with $C_{2,l}^{\bullet}$ and $C_{3,l}^{\bullet}$. The shaded regions represent different rooted connected copies.}
\label{graphic1}
\end{figure}

Following the notation in Figure~\ref{graphic1}, by writing $I=(k_1,l_1,k_2,l_2,k_3,l_3)$ we have the following relations:
\begin{eqnarray*}
C_{2,l}^{\bullet}(x,u)=\frac{1}{2} x u^6 \sum_{I}^{\ast} u^{k_1k_2+k_2k_3+k_2k_3+4(k_1+k_2+k_3)+3(l_1+l_2)+2l_3} C_{k_1,l_1}^{\bullet}(x,u)C_{k_2,l_2}^{\bullet}(x,u)C_{k_3,l_3}^{\bullet}(x,u),\\
C_{3,l}^{\bullet}(x,u)=\frac{1}{2} x u^6 \sum_{I}^{\ast\ast} u^{k_1k_3+k_2k_3+4(k_1+k_2+k_3)+2(l_1+l_2)+3l_3} C_{k_1,l_1}^{\bullet}(x,u)C_{k_2,l_2}^{\bullet}(x,u)C_{k_3,l_3}^{\bullet}(x,u),
\end{eqnarray*}
where the first sum $\sum_{I}^{\ast}$ is taken over indices $I$ satisfying  $k_1+k_2+4=l$ and the second sum $\sum_{I}^{\ast\ast}$ is taken over indices $I$ satisfying  $k_1+k_2+k_3+4=l$.
Both formulas are easily explained by checking the structure depicted in Figure \ref{graphic1}.

It is very important to notice that the function $f(b_i^{\circ}, I)$ depends on the choice of the block.
In fact, this term not only encodes an internal information of the block, but also how different pasted connected copies interact along it in order to create new copies of $P_3$.
Let us describe more precisely the leftmost term defining $C_{2,l}^{\bullet}(x,u)$. An object counted in $C_{2,l}^{\bullet}(x,u)$ is obtained by pasting three rooted connected objects over vertices of a block of type $b_1^{\circ}$.
Then, the final number of paths of length $3$ arise from the following contributions:
\begin{enumerate}
\item The existing paths of length $3$ in each of the pasted rooted connected components.
\item The existing paths of length $3$ in $b_1^{\circ}$ (6 in total).
\item Paths that are created by concatenating paths of length 1 in $b_1^{\circ}$ with paths of length 2 in each pasted rooted connected component.
\item Paths that are created by concatenating paths of length 2 in $b_1^{\circ}$ with paths of length 1 in each pasted rooted connected component.
\item Paths created by using 2 paths of length 1 in a pair rooted connected components which are linked in $b_1^{\circ}$ by a path of length 1.
\end{enumerate}
As mentioned above, the most difficult term to be encoded is the one in item (5) and it is given by the correlation term $k_1k_2+k_2k_3+k_2k_3$, which is build explicitly using the internal structure of $b_1^\circ$ and the set of indices $I$.

The situation is even more involved if several blocks are attached to the
root. For example, the equations for $C_{4,l}^{\bullet}(x,u)$ and $C_{5,l}^{\bullet}(x,u)$ require
the whole information of the two attached blocks. Nevertheless, it is clear how to set up
an infinite system of equations for the functions $C_{k,l}^{\bullet}(x,u)$.

Unfortunately this system does not satisfy all assumptions of Theorem~\ref{Thanalytic2}.
Namely if we set $u=1$ we obtain for example
\begin{eqnarray*}
C_{2,l}^{\bullet}(x,1)=\frac{1}{2} x  \sum_{I}^{\ast} C_{k_1,l_1}^{\bullet}(x,1)C_{k_2,l_2}^{\bullet}(x,1)C_{k_3,l_3}^{\bullet}(x,1)
\end{eqnarray*}
where thus sum is taken over indices $I$ satisfying $k_1+k_2+4=l$. This means that the right hand sind cannot
be written in terms of $C^\bullet(x) = \sum_{k,l} C_{k,l}^{\bullet}(x,1)$.

However, it is possible to modify our setting slightly. Instead of analyzing the block decomposition related
to the equation $C^\bullet(x) = x \exp( B^\circ(C^\bullet(x)))$ we iterate this equation and replace it
by
\[
C^\bullet(x) = x \exp( B^\circ(x \exp( B^\circ(C^\bullet(x)))))),
\]
which means that we specify first a tree of height two of (rooted) blocks before we substitute each vertex
by $C^\bullet$ in order to obtain a recurvice description for $C^\bullet$.

We demonstrate this procedure by considering one special instance that is part of the equation
for $C_{2,6}^\bullet(x,u)$, compare with Figure~\ref{graphic2}. Here the root block is of type
$b_1^\circ$. One non-root vertex of this block is attached by another block of type $b_1^\circ$,
a second non-root vertex is attached by a block of type $b_2^\circ$, whereas the third
non-root vertex has no block attached. It is clear that such a block structure will lead to a
connected graph of type $(k,l) = (2,6)$ - and there are 5 other instances similar to that which
cover then all situations of this form.

\begin{figure}[htb]
\begin{center}
\includegraphics[width=7.0 cm]{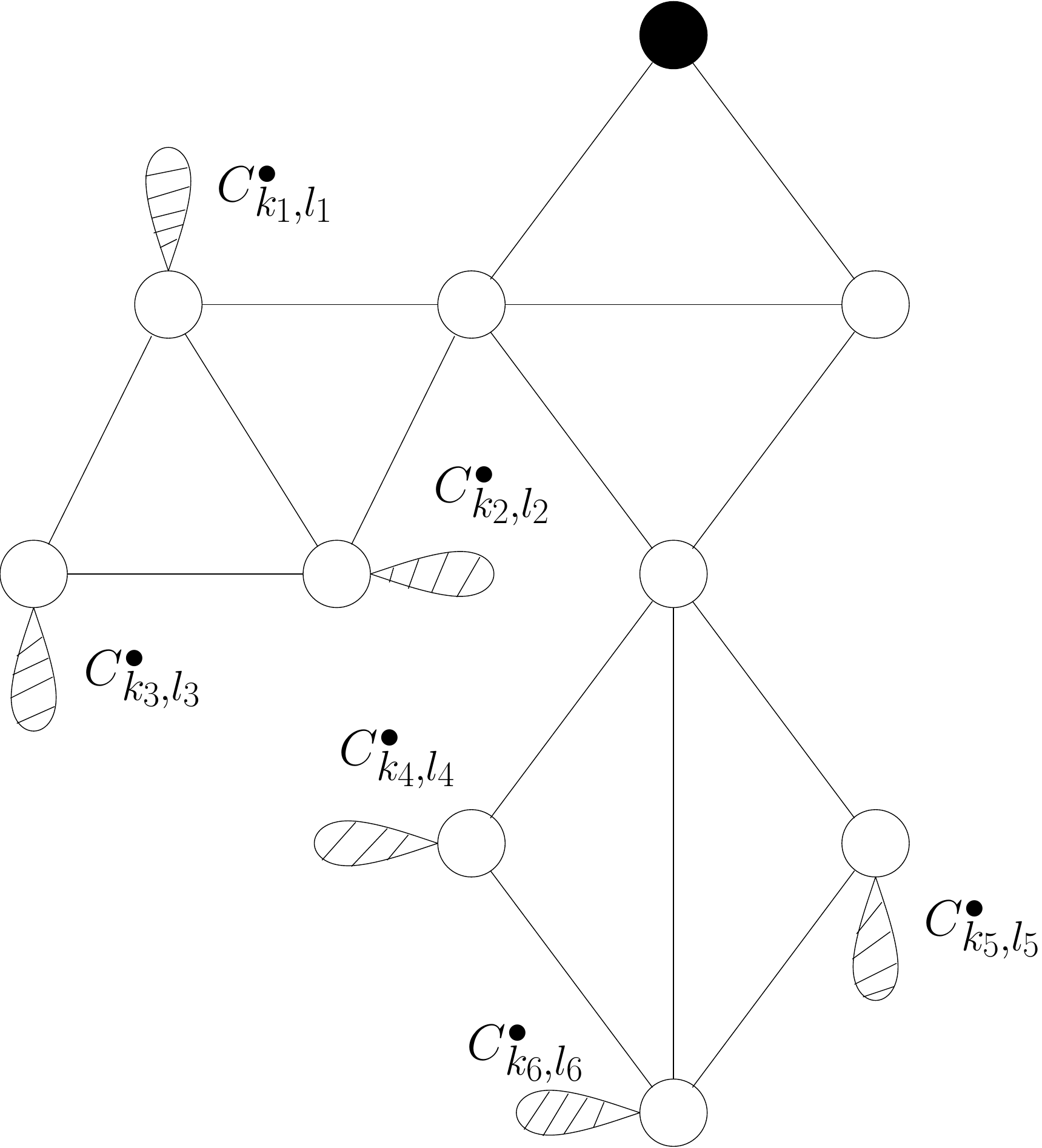}
\end{center}
\caption{ Construction of rooted connected objects of type  $(k,l) = (2,6)$.}
\label{graphic2}
\end{figure}

The corresponding generating function is then of the form
\begin{eqnarray*}
\frac 98 x^{10} u^{64} \sum_K u^{H(K)} \prod_{i=1}^6 C_{k_i,l_i}^{\bullet}(x,u),
\end{eqnarray*}
where the sum is taken over all indices
$K = (k_1,l_1,k_2,l_2,k_3,l_3,k_4,l_4,k_5,l_5,k_6,l_6)$ and
\begin{eqnarray*}
H(K) &=& k_1k_2+ k_1k_3+k_2k_3 + 3(l_1+l_2) + 2l_3 + 4(k_1+k_2+k_3) \\
&+& k_4k_5+ k_4k_6+k_5k_6 + 2(l_4+l_5) + 3l_6 + 4(k_4+k_4+k_4).
\end{eqnarray*}
In the same way we can deal with the other 5 cases which leads to a functional equation
for $C_{2,6}^\bullet(x,u)$ of the form
\[
C_{2,6}^\bullet(x,u) = F_{2,6}\left( x,u, (C_{k,l}^\bullet(x,u))_{k,l\ge 0}\right).
\]
The main difference between this representation and the above version is that the sum
over the indices $K$ has no restriction which implies that
\[
C_{2,6}^\bullet(x,1) = \frac 98 x^{10} C^\bullet(x)^6.
\]
Hence we are again in a situation, where we can apply Theorem~\ref{Thanalytic2} which
leads to a central limit theorem.

\subsection{The general case. Proof of Theorem \ref{thm:main}}\label{subs.general}

We finally deal with the study of the number of copies of a fixed subgraph $H$.
Recall that the new difficulty emerging when considering copies of $P_3$  was the existence of a correlation
between the root type and the root types of the attached connected graphs
In this section we show how we can overcome these (and other) kinds of problems. We start
with the observation that the equation characterizing (rooted) connected graphs in
terms of blocks can be iteratively written as follows:
\begin{equation}\label{eq:rec}
C^{\bullet}(x)=x \exp(B^{\circ}(C^{\bullet}(x)))=x \exp(B^{\circ}(x \exp(B^{\circ}(C^{\bullet}(x)))))=\dots.
\end{equation}
When stopping after $h$ iterations, Equation~\eqref{eq:rec} says that a rooted connected graph is obtained by repeating $h$ times the operation of pasting a set of rooted blocks on vertices, and finally substituting recursively rooted connected graphs on each vertex -- in the previous section we did just one interation.

We introduce now some notation.
Let $c^{\bullet}\in \mathcal{C}^{\bullet}$ be a rooted graph in our graph class.
%Recall that its block graph is a tree.
We say that the set of blocks of $c^{\bullet}$ which are at distance at most $h$ to the root vertex in its block graph is the $h$-\emph{root block} of $c^{\bullet}$.
We define $\mathcal{B}^{(h),\circ}$ to be the family of all possible $h$-root blocks.
We write $B^{(h),\circ}(x,w)$ for the EGF associated to $\mathcal{B}^{(h),\circ}$, where $x$ encodes vertices on the $h$-root block until level $h-1$, while the extra parameter $w$ encodes vertices belonging to the blocks pasted in the last step of the iteration (namely, at level $h$). Then, it is satisfies
\begin{equation*}
B^{(h),\circ}(x,w)=\exp\left(x B^{(h-1),\circ}(x,w)\right)
\end{equation*}
with initial condition $B^{(1),\circ}(x,w)=\exp(B^{\circ}(w))$.
From Equation~\eqref{eq:rec} we get that for each $h\geq 1$, $C^{\bullet}(x)= x B^{(h),\circ}(x , C^{\bullet}(x))$.
In particular $C^{\bullet}(x)= x B^{(1),\circ}(x,C^{\bullet}(x))=x \exp(B^{\circ}(C^{\bullet}(x)))$.
According to the previous considerations it is obvious that the composition $B^{(h),\circ}(C^{\bullet}(x))$ is also subcritical.
Hence, in we may assume that all the analysis will be done for points $x$ where the function $B^{(h),\circ}(x)$ is analytic.
We also write $B^{(h),\circ}(x,w,u)$ for the counting formula of $h$-root blocks, where $u$ marks copies of the subgraph $H$.
%Observe also that $$B^{(h),\circ}(x,w,u)\neq x\exp\left(B^{(h-1),\circ}(x,w,u)\right),$$

%as when passing from the $(h-1)$-root block to the $h$-root block it is possible that we create extra copies of $H$.

Let us now study substructures of $H$ that will be necessary for the encoding.
Assume that the block graph of $H$ has diameter $h$.
The main observation we exploit is that all copies of $H$ which are incident to the root vertex of $c^{\bullet}$ are contained in the $h$-root block of $c^{\bullet}$.
Let $\mathcal{H}_0=\{H_1,\dots, H_s\}$ be all the connected subgraphs spanned by subsets of blocks of $H$.
For a given $H_i\in \mathcal{H}_0$ we denote by $\overline{H_i}$ the set of blocks in $H$ not contained in $H_i$.
Given $H_i\in \mathcal{H}_0$ we say that a vertex $v$ in $H_i$ is a \emph{virtual cut-vertex} if it is either a cut-vertex in $H_i$, or when we embed $H_i$ in $H$, the resulting vertex becomes a cut-vertex in the ambient graph $H$.
See Figure \ref{subsubgraph} for an example of a subgraph with four virtual cut-vertices.

\begin{figure}[htb]
\begin{center}
\includegraphics[width=7.4 cm]{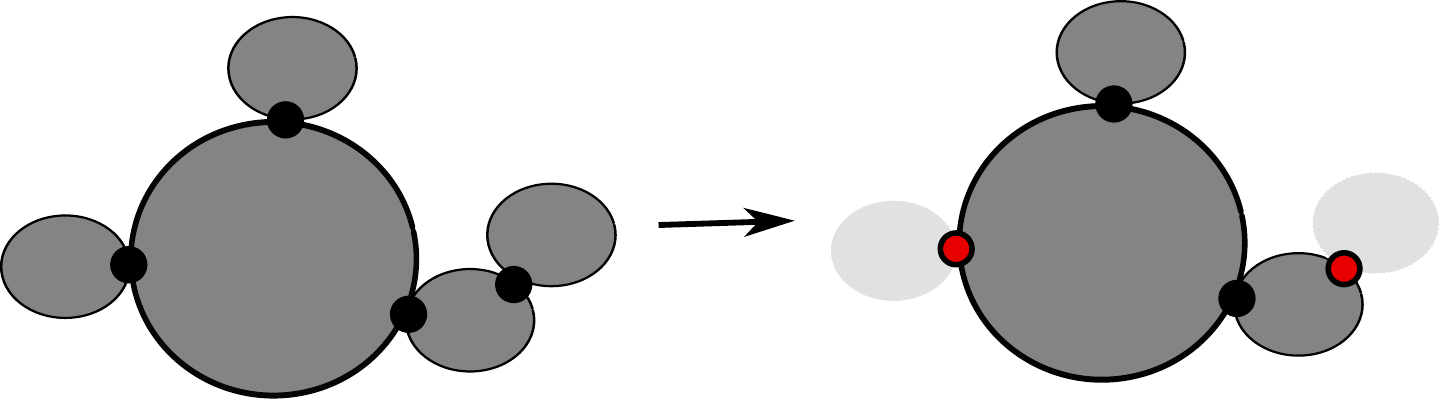}
\end{center}
\caption{The graph $H$ and a subgraph of $H$ with four virtual cut-vertices (two of them, in red, are not cut-vertices in the subgraph).}
\label{subsubgraph}
\end{figure}

We denote then by $\mathcal{H}$ the family of graphs constructed from the graphs in $\mathcal{H}_0$ by rooting one of its virtual cut-vertices. Let $d$ denote the cardinality of $\mathcal{H}$ and let $\mathcal{I} = \mathbb{N}^d$
be the set of $d$-dimensional indices $I$.
%We denote by $\mathcal{H}_k$ the family of graphs constructed from the graphs in $\mathcal{H}$ by rooting $k$ (different) virtual cut-vertices in elements in $\mathcal{H}$.
%%In particular, $|\mathcal{H}| \leq 2^{|H|}$ and $|\mathcal{H}_k| \leq |H|^{k} 2^{|H|}$.
%As in the case of elements of $\mathcal{H}_0$, for $H_i^{\bullet} \in \mathcal{H}_k$, we write $\overline{H_i^{\bullet}}$ for the set of blocks in $H$ not contained in $H_i^{\bullet}$.
%Additionally, $k$ must be smaller than $|H|$, otherwise we would have repeated rooted vertices in the graphs in $\mathcal{H}_k$.
%We finally write $\mathcal{H} = \bigcup_{k} \mathcal{H}_k$, which is a finite set.
For every $I = (i_1,\ldots,i_d) \in \mathcal{I}$ (which we also call {\it profile}) we consider the combinatorial family $\mathcal{C}_{I}^{\bullet}$ (with exponential generating function  $C_{I}^{\bullet}(x,u)$) of rooted connected graphs in $\mathcal{C}^{\bullet}$ with $i_j$
copies of the $j$-th subgraph of $\mathcal{H}$, $1\le j\le d$, where the virtual cut-vertex coincides with the root vertex of the connected graph in $\mathcal{C}_{I}^{\bullet}$.
%We use as many indices as the number of different virtual cut-vertices of the structure.
%In other words, the profile encodes the number of appearances of $H_i\in \mathcal{H}$ making coincide (in all possible ways) virtual cut-vertices with the root vertex of the structure.
Similarly, we define $\mathcal{B}_{I}^{(h),\circ}$ the family of $h$-rooted blocks whose profile is equal to $I$.
Hence, $\mathcal{B}^{(h),\circ}= \bigcup_{I\in \mathcal{I}} \mathcal{B}_{I}^{(h),\circ}$.
%Roughly speaking, we need to recall partial structure of $H$ on the corresponding $h$-root blocks in order to build all possible copies of $H$.

Let $c^{\bullet}\in \mathcal{C}_{I}^{\bullet}$. There are three different types of copies of $H$ in $c^{\bullet}$:
\begin{enumerate}
\item [\textbf{Case $(b)$}] Copies of $H$ already existing on the $h$-root block of $c^{\bullet}$.

\item [\textbf{Case $(c)$}] Copies of $H$ already existing on the rooted connected graphs that we attach at the $h$-root block of $c^{\bullet}.$

\item [\textbf{Case $(bc)$}] Copies created by using some subgraph of $H$ from the $h$-root block of $c^{\bullet}$ and completing it to $H$ by using attached rooted connected graphs with convenient profiles.
\end{enumerate}
See Figure~\ref{graphic3} for an example of a subgraph $H$ with $h=2$, and three different copies of $H$ arising from these 3 different sources.
\begin{figure}[htb]
\begin{center}
\includegraphics[width=10.3 cm]{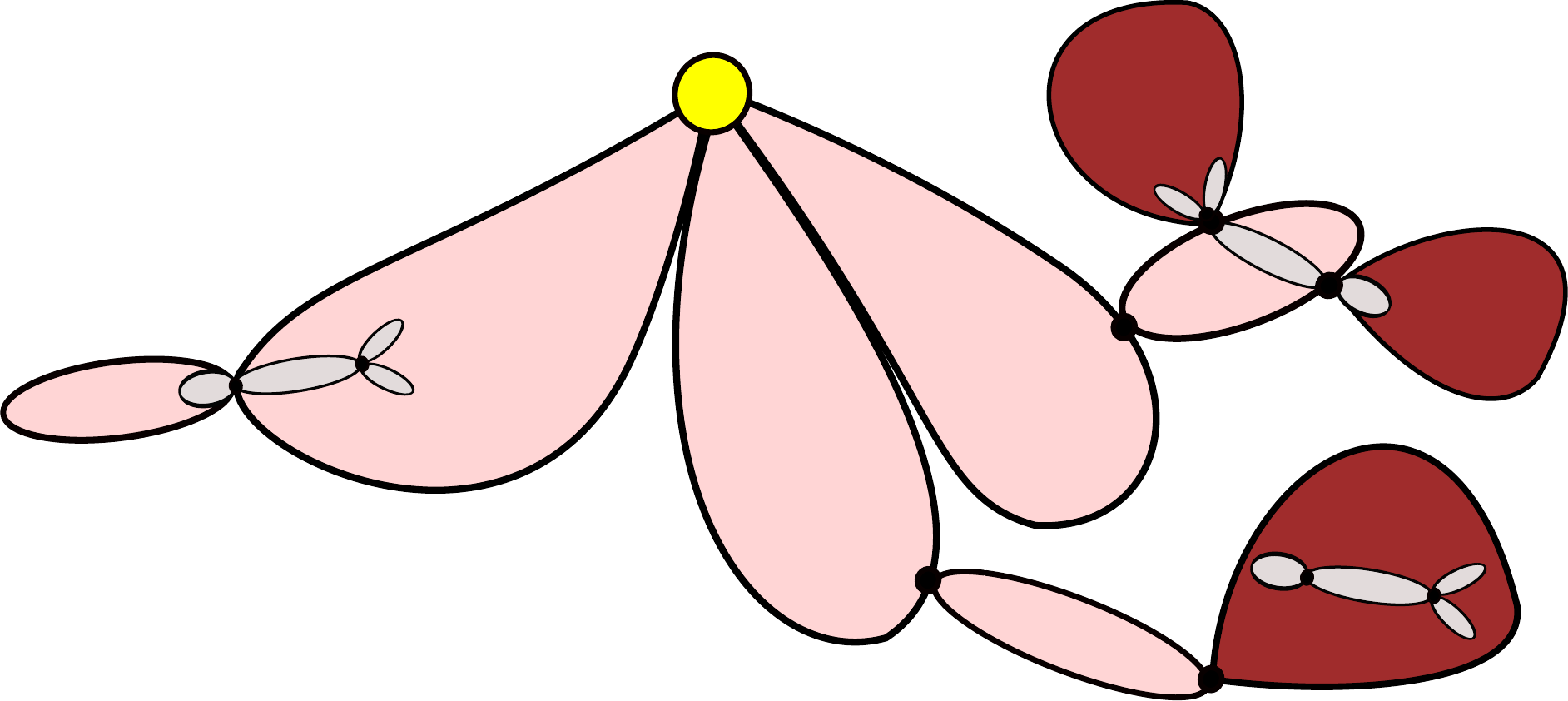}
\end{center}
\caption{ A generic copy of $H$ in the construction may arise from three different sources. The colour of the root block and the rooted connected graphs attached to it is different.}
\label{graphic3}
\end{figure}

We can now write an expression for $C^{\bullet}_{I}(x,u)$.
Let $b_h$ be a $h$-root block.
Denote by $|b_h|_2$ the number of vertices of $b_h$ on the $h$-level of the root block, and $|b_h|_1=|b_h|-|b_h|_2$.
We write $\textbf{I}(|b_h|_2)=(I_1,\dots,I_{|b_h|_2})$.
This set of profiles will be the ones of the rooted connected graphs that we will attach to each vertex of the $h$-level of the $h$-root block.
Also, when fixed $b_h$ and a set of profiles $\mathbf{I}(|b_h|_2)$, we write
\begin{itemize}
\item $G_1(b_h)$ for the number of copies of $H$ in Case $(b)$.
\item $G_2(b_h,\textbf{I}(|b_h|_2))$ for the number of copies of $H$ in Case $(bc)$.
\end{itemize}
Observe that both $G_1(b_h)$ and  $G_2(b_h,\textbf{I}(|b_h|_2))$  depend on the specific structure of $b_h$ (and also on the set of the profiles $\textbf{I}(|b_h|_2)$ in $G_2$).
With this terminology in mind, now it is easy to write an equation for $C^{\bullet}_{I}(x,u)$:
\begin{equation}\label{eq:Main}
C^{\bullet}_{I}(x,u)= x \sum_{b_{h}\in \mathcal{B}^{(h),\circ}_I} \sum_{\textbf{I}(|b_h|_2)} u^{G_1(b_h)+G_2(b_h,\textbf{I}(|b_h|_2))}\frac{x^{|b_h|_1}}{|b_h|_1 !}\prod_{i=1}^{|b_{h}|_2} C^{\bullet}_{I_i}(x,u),
\end{equation}
where the second sum is extended to all possible sets of $|b_h|$ profiles.
We are now ready to prove Theorem \ref{thm:main} by analyzing Equation \eqref{eq:Main}.
We write $\mathbf{y}=(C^{\bullet}_{I}(x,u))_{I}$, which is a solution to the infinite system of equations $y_I=F_I(x,\mathbf{y},u)$ with
$$
y_I=F_I(x,\mathbf{y},u)=x \sum_{b_{h}\in \mathcal{B}^{(h),\circ}_I} \sum_{\textbf{I}(|b_h|_2)}  u^{G_1(b_h)+G_2(b_h,\textbf{I}(|b_h|_2))} \frac{x^{|b_h|_1}}{|b_h|_1 !} \prod_{j=1}^{|b_{h}|_2} y_{I_j}.
$$
We can now check that this system of equations satisfies the conditions of Theorem \ref{Thanalytic2}.
We may assume in all the analysis that all variables $y_I$, $x$ are positive.
Let us start with the partition property.
By writing $u=1$, we get that $F_I(x,\mathbf{y},1)$ is equal to
\begin{equation}\label{eq:Main}
F_I(x,\mathbf{y},1)= x \sum_{b_{h}\in \mathcal{B}^{(h),\circ}_I} \frac{x^{|b_h|_1}}{|b_h|_1 !}  \prod_{j=1}^{|b_{h}|_2}y_{I_j}= x B^{(h),\circ}_{I}\left(x,\sum_{J}y_J\right).
\end{equation}
Hence, $F_I(x,\mathbf{y},1)$ is equal to $x B^{(h),\circ}_{I}\left(x,\sum_{J}y_J\right)=\widetilde{F}_{I}\left(x,\sum_{J}y_J\right)$, and Condition~\eqref{eqinfsystemcond1} is satisfied.
Let us check now Condition ~\eqref{eqinfsystemcond2}. Observe that
$$\sum_{I}\widetilde{F}_{I}(x,y)=\sum_{I} x B^{(h),\circ}_{I}\left(x,y\right)= x B^{(h),\circ}(x,y)=F(x,y),$$
which is analytic in $x$ and $y$ due to the subcritical condition (recall that with this notation, $C^{\bullet}(x)=F(x,C^{\bullet}(x))$).
Also, the condition assuring that this system of equations is well defined and analytic in the functional space considered in Remark \ref{mainremark2} is satisfied by taking a sufficiently large (but bounded) number of derivatives of $\widetilde{F}_{I}(x,y)$ with respect to $y$. Again, by the subcritical condition all these derivatives are bounded and consequently, for each choice of $m\geq 1$
$$\sum_{{I}} \|{I}\|_1^m \widetilde{F}_I(x,y) \leq \left(xy\frac{\partial^2}{\partial x \partial y}\right)^{m f_1(|H|)} F(x,y),$$
for a certain function $f_1$ that only depends on the size of $|H|$ (and hence, it is bounded).
This fact finally proves the first part of the conditions.
Let us show now Condition \eqref{eqderivatives}. We only argue the case $r=1$, as the arguments for the second and the third derivatives are very similar (but much longer).
We will show that the terms can be bounded by a constant number of derivatives (depending on $|H|$) of an analytic function, hence the resulting value will be bounded as well.
We need first to bound the following derivative at $t=0$:
\begin{align*}\left|\frac{\partial }{\partial t}F_I(x,\mathbf{y},e^{it})\right|&=\left|x \frac{\partial }{\partial t}  \sum_{b_{h}\in \mathcal{B}^{(h),\circ}_I} \sum_{\textbf{I}(|b_h|_2)}  e^{it(G_1(b_h)+G_2(b_h,\textbf{I}(|b_h|_2)))} \frac{x^{|b_h|_1}}{|b_h|_1 !}  \prod_{j=1}^{|b_{h}|_2} y_{I_j}\right|\\
&= x  \sum_{b_{h}\in \mathcal{B}^{(h),\circ}_I} \sum_{\textbf{I}(|b_h|)}  (G_1(b_h)+G_2(b_h,\textbf{I}(|b_h|_2))) \frac{x^{|b_h|_1}}{|b_h|_1 !} \prod_{j=1}^{|b_{h}|} y_{I_j}
\end{align*}
Hence we have two different contributions, namely expressions $G_1(b_h)$ and $G_2(b_h,\textbf{I}(|b_h|_2))$.
Observe first that $G_1(b_h)$ counts the number of copies of $H$ in $b_h$, hence it is bounded by $O(|b_h|^{|H|})$. Consequently, we have the bound
$$x  \sum_{b_{h}\in \mathcal{B}^{(h),\circ}_I} \sum_{\textbf{I}(|b_h|)} G_1(b_h) \frac{x^{|b_h|_1}}{|b_h|_1 !}  \prod_{j=1}^{|b_{h}|} y_{I_j}\leq \left(x y\frac{\partial^2}{\partial x \partial y}\right)^{|H|} \widetilde{F}_I(x,y) $$
which is bounded, and hence
\begin{align*}
&x \sum_{I} \|I\|_1^m  \sum_{b_{h}\in \mathcal{B}^{(h),\circ}_I} \sum_{\textbf{I}(|b_h|_2)} G_1(b_h) \frac{x^{|b_h|_1}}{|b_h|_1 !}  \prod_{j=1}^{|b_{h}|} y_{I_j}\\
&\leq x \sum_{{I}} \|{I}\|_1^m \left(x y\frac{\partial^2}{\partial x \partial y}\right)^{|H|} \widetilde{F}_I(x,y) \leq x\left(y\frac{\partial^2}{\partial x \partial y}\right)^{mf_1(|H|)+|H|} F(x,y).
\end{align*}
It finally remains to study the contribution $G_2(b_h,\textbf{I}(|b_h|_2))$, which is the number of copies of $H$ created in Case $(bc)$.
We can obtain a bound for $G_2(b_h,\textbf{I}(|b_h|))$ by using that the size of $H$ is bounded by $|H|$.
Observe that any copy created in Case $(bc)$ (and hence counted by $G_2$) is obtained by taking a subgraph of $H$ in the $h$-root block, and completing it to $H$ by attaching at most $|H|$ substructures arising from pending connected graphs.
The number of subgraphs of $|H|$ in $b_h$ is bounded by $|b_h|^{f_2(|H|)}$, for a certain function $f_2$.
This means that
$$G_2(b_h,\textbf{I}(|b_h|_2))\leq |b_h|^{f_2(|H|)} \sum_{\ast} \|I_{j_1}\|_1 \dots \|I_{j_{|H|}}\|_1,$$
where the sum in the previous expression is extended to all subsets of size $|H|$ of $\{1,\dots, |b_h|\}$.
Observe that the total number of sum terms is bounded then by $|b_h|^{|H|}$.
Putting now all together we get the following:
\begin{align*}
& x  \sum_{b_{h}\in \mathcal{B}^{(h),\circ}_I} \sum_{\textbf{I}(|b_h|_2)} G_2(b_h;\textbf{I}(|b_h|_2))\frac{x^{|b_h|_1}}{|b_h|_1 !} \prod_{j=1}^{|b_{h}|_2} y_{I_j} \\
& \leq x \left(\sum_{I}\|I\|_1|y_I|\right)^{|H|}\sum_{b_{h}\in \mathcal{B}^{(h),\circ}_I} |b_h|^{f_2(|H|)} |b_h|^{|H|} \frac{x^{|b_h|_1}}{|b_h|_1 !}   \prod_{j=1}^{|b_h|} y_{I_j} \\
&\leq \left(\sum_{I}\|I\|_1|y_I|\right)^{|H|} \left(xy\frac{\partial^2}{\partial x \partial y}\right)^{f_2(|H|)+|H|} \widetilde{F}_I(x,y).
\end{align*}
By assumption, the sum $\sum_{I}\|I\|_1|y_I|$ is bounded, hence the previous term is bounded as well.
Finally we can get bounded expressions for the weighted sum with coefficients $\|I\|_1^m$, as we did when analyzing the function $G_1(b_h)$. This concludes the study for the first derivative. As mentioned, case $r=2$ and $r=3$ can be similarly handled and obtain similarly bounded expressions.

This concludes the proof of Theorem \ref{thm:main}.
%%%%%%%%%%%%%%%%%%%%%%%%%%%%%%%%%%%%%%%%%%%%%%%%%%%%%%%%%%%%%%%%%%%%%%%%%%%%%%%%%%%%%%%%%%%%%%%%%%%%%%%%%%%%%%%%%%%%%%%%%%%%%%%%%%%%%%%%%%%%%%%%%%%%
%%%%%%%%%%%%%%%%%%%%%%%%%%%%%%%%%%%%%%%%%%%%%%%%%%%%%%%%%%%%%%%%%%%%%%%%%%%%%%%%%%%%%%%%%%%%%%%%%%%%%%%%%%%%%%%%%%%%%%%%%%%%%%%%%%%%%%%%%%%%%%%%%%%%

\section{Computations}\label{sec:comp}

In this section we computationally analyze the statistics for some small subgraphs in series-parallel graphs (also written as SP graphs).
In particular, we compute the subgraph statistics for triangles in a uniformly at random 2-connected and connected SP graph of size $n$. We also compute the subgraph statistics for triangles in a uniformly at random 2-connected SP graph of size $n$,
which is technically different to the case of triangles, but
the connected case is analogous and we skip it.
Additionally, our methodology gives easily the asymptotic enumeration of SP graphs avoiding the subgraph under consideration.
In this prominent case, it is straightforward to apply Remark \ref{rem:sigma} in order to justify that the corresponding constant $\sigma_H^2>0$.
Hence, for all subgraphs in the connected level the second statement in the Main Theorem \ref{thm:main} will hold.

In order to analyze SP graphs we use a variant of Tutte's decomposition into 3-connected components, as depicted in~\cite{tutte1966}.
Recall that this strategy is used when a class of graphs satisfies that a graph belongs to the family if and only if its connected, 2-connected and 3-connected components also belong to, as it is the case in SP graphs.

As already mentioned, a connected graph is obtained from its tree decomposition into 2-connected blocks.
A 2-connected graph is decomposed into 3-connected graphs using networks to join the pieces.
In the case of SP graphs there are no 3-connected graphs, so we start with networks as the basic building blocks.
The key point here is that networks are easy enough to be built, so we can control the appearance of simple structures, like cycles.
If these structures are 2-connected then they can only appear inside 2-connected blocks, so if we count them at the 2-connected level then Tutte's decomposition gives the total number for general graphs.

In one of the steps of the decomposition we have to obtain a 2-connected graph from a network. In general, a network is obtained by picking an edge of the 2-connected graph and performing some minor corrections.
Therefore, in order to obtain a 2-connected graph from a network we have to 'forget' a root edge.
Since we can translate the action of rooting an edge in terms of generating functions as differentiating with respect to the variable
that counts edges, we can translate the opposite action (forgetting the root) as the integration with respect to the same variable.
This was done in~\cite{Gimenez2009} to obtain the generating function of 2-connected planar graphs.
However, we will use a more recent approach, purely combinatorial, defined following the ideas of the grammar developed in~\cite{Chapuy2008}.
This approach uses the so-called Dissymmetry Theorem for trees~\cite{BerLaLe98}.
This technique gives a bijection that relates unrooted trees and trees rooted in both a vertex and an edge, which is used to express the generating function of unrooted trees in terms of the
generating function of rooted trees.
In~\cite{Chapuy2008} the authors consider the decomposition of a 2-connected graph into networks.
Since the class is tree-decomposable, they show that the dissymmetry theorem
can be used to obtain the generating function of 2-connected graphs
in terms of the generating function of the networks, with no integration involved.

Note that in Section~\ref{sec:2-con} we already prove that the number of copies of a 2-connected subgraph in a connected graph is normally distributed.
In this section we prove the same for the 2-connected level in the particular class of SP graphs.
Moreover, we give exact computations of the parameters of the Gaussian laws by means of the Quasi-powers Theorem.

This section is divided in the following way: in subsection \ref{ss.triangles-SP} both the number of copies of triangles and series-parallel without triangles are studied. Later, in subsection \ref{ss.4-cycles-sp} the subgraph under study is the cycle of length four. Finally, in subsection \ref{ss.girth} equations to study series-parallel graphs with given girth are shown.

\subsection{Triangles in series-parallel graphs}\label{ss.triangles-SP}

Since there are no 3-connected SP graphs, we start by computing the generating functions $D^\blacktriangle(x,y,u)$ of networks, where $x,\, y$ mark vertices and edges, respectively. We add the additional parameter $u$ which is used to encode triangles.

Recall that a network is obtained from a 2-connected series-parallel graph by choosing and orienting an edge. It might not occur in the graph, and the vertices incident to it, the poles, are not labelled, but instead one of them is consider to be 0, and the other one is $\infty$.
For convenience we split both series and parallel generating functions as follows. We define $P^\blacktriangle_0:=P^\blacktriangle_0(x,y,u)$ as the generating function of parallel networks that do not contain an edge between the poles, whereas $P^\blacktriangle_1:=P^\blacktriangle_1(x,y,u)$ is the generating function of parallel networks where there is an edge connecting the poles.
For convenience, we include the network consisting of a single edge in $P^\blacktriangle_1$.

We define $S^\blacktriangle_2:=S^\blacktriangle_2(x,y,u)$ as the generating function of series networks where there is a path of length exactly 2 between the poles, or equivalently
where there exists a single cut vertex.
We also define $S^\blacktriangle_3:=S^\blacktriangle_3(x,y,u)$ as the remaining series networks. Namely, the ones where the graph distance between the poles is at
least 3.
The generating function $D^\blacktriangle:=D^\blacktriangle(x,y,u)$ can be expressed then as the solution of the following system of equations:
\begin{align} \label{eq:D}
D^\blacktriangle &= P^\blacktriangle_0+P^\blacktriangle_1+S^\blacktriangle_2+S^\blacktriangle_3\\\nonumber
P^\blacktriangle_0 &= \exp_{\geq 2} (S^\blacktriangle_2+S^\blacktriangle_3)\\\nonumber
P^\blacktriangle_1 &= y\exp (uS^\blacktriangle_2+S^\blacktriangle_3) \\\nonumber
S^\blacktriangle_2 &= x(P^\blacktriangle_1)^2\\\nonumber
S^\blacktriangle_3 &= xD^\blacktriangle P^\blacktriangle_0 + xP^\blacktriangle_1(P^\blacktriangle_0+S^\blacktriangle_2+S^\blacktriangle_3).\nonumber
\end{align}

A graph in $P^\blacktriangle_0$ is obtained as a set
of at least two
series graphs in parallel, since no series graph has
an edge between the poles. A graph in $P^\blacktriangle_1$
is obtained by putting a set of series graphs in parallel
with an edge. Note that each series graph that contains
a path of length 2 between the poles will produce a triangle.
A graph in $S^\blacktriangle_2$ has a single cut vertex,
and an edge joining it to both poles, which might be
in parallel with other series graphs, so we need two
copies of $P^\blacktriangle_1$.
A graph in $S^\blacktriangle_3$ has at least one cut vertex.
Let $x$ be the cut vertex closest to pole $0$. There are
two options: either $x$ is joined to pole $0$ by an edge,
and therefore by a graph in $P^\blacktriangle_1$,
or it is joined to pole $0$ by a graph in $P^\blacktriangle_0$.
In the former case, there cannot be and edge between
$x$ and pole $\infty$, so there must be a graph in
$P^\blacktriangle_0$, $S^\blacktriangle_2$ or
$S^\blacktriangle_3$ that joins $x$ and pole $\infty$.
In the latter case any network is possible, since the
distance between the poles will be at least 3.

From these equations we deduce that triangles can only come up from parallel constructions where the poles are connected by an
edge and at least one path of length 2.
The generating function $D^\blacktriangle(x,y,u)$ cannot be expressed in terms of elementary functions,
but we can obtain the first terms of its expansion near 0:
\begin{align*}
D^\blacktriangle(x,y,u) &=
y+ \left( {y}^{2}+u{y}^{3} \right) x+ \left( 2y^{3}+3y^{4}+4uy^{4}+5u^{2}y^{5}\right) {{x}^{2}\over 2!}\\
&+ \left( 6y^{4}+30y^{5}+7y^{6}+18uy^{5}+48uy^{6}+ 36u^2y^{6}+{49
}u^{3}y^{7}\right) {{x}^{3}\over 3!}+O(x^4).
\end{align*}
Now that we know $D^\blacktriangle$ and the auxiliary functions
$P^\blacktriangle_0$, $P^\blacktriangle_1$, $S^\blacktriangle_2$ and $S^\blacktriangle_3$, we can use the
Dissymmetry Theorem for trees in order to obtain the generating function $B^\blacktriangle(x,y,u)$
of 2-connected SP graphs, where $x,\, y,\, u$ marks
vertices, edges and triangles, respectively.
We will use the same approach as in~\cite{Chapuy2008}:
since the class of 2-connected SP graphs is tree-decomposable,
we can apply the following bijection.
$$
B^\blacktriangle ~ + ~
B^\blacktriangle_{\circ \rightarrow \circ}
\simeq B^\blacktriangle_{\circ} ~ + ~
B^\blacktriangle_{\circ - \circ},
$$
where $B^\blacktriangle_{\circ}$
represents the class of 2-connected SP graphs with
a distinguished vertex in the tree decomposition,
i.e., either a ring or a multiedge,
$B^\blacktriangle_{\circ - \circ}$
represents the class of 2-connected SP graphs with
a distinguished edge in the tree decomposition,
which must be incident to both a ring and a multiedge,
and $B^\blacktriangle_{\circ \rightarrow \circ}$
represents the class of 2-connected SP graphs with
a distinguished oriented edge in the tree decomposition.
This leads to the following expressions:
$$
B^\blacktriangle_R(x,y,u)= \text{Cyc}(x(P^\blacktriangle_0+P^\blacktriangle_1))+(u-1){(xP^\blacktriangle_1)^3\over 6}
$$
$$
B^\blacktriangle_M(x,y,u)= {x^2\over 2}
\left(y \exp_{\geq 2}(uS^\blacktriangle_2+S^\blacktriangle_3)+\exp_{\geq 3}(S^\blacktriangle_2+S^\blacktriangle_3)\right)
$$
$$
B^\blacktriangle_{MR}(x,y,u) = {x^2\over 2} \left((S^\blacktriangle_2+S^\blacktriangle_3)(P^\blacktriangle_0+P^\blacktriangle_1-y)+(u-1)(P^\blacktriangle_1-y)S^\blacktriangle_2\right),
$$
where $B^\blacktriangle_R$
represents the class of 2-connected SP graphs with
a distinguished ring in the tree decomposition,
$B^\blacktriangle_M$ represents the class of 2-connected SP graphs with
a distinguished multiedge in the tree decomposition,
and $B^\blacktriangle_{MR}$
represents the class of 2-connected SP graphs with
a distinguished pair of incident ring and multiedge.
In the case of $B^\blacktriangle_R$ we have to consider
the special case where the ring is of length three, and
the parallel networks that replace the edges of the ring
are of the kind $P^\blacktriangle_1$, since this
generates a new triangle, as it is shown in Figure~\ref{graphicring}.
In the case of $B^\blacktriangle_M$ we distinguish 2 cases,
depending on whether one of the edges of the multiedge
is not replaced with a series network, but with an edge,
since this generates a new triangle for every other edge
replaced with a series network in $S^\blacktriangle_2$.
In the case of $B^\blacktriangle_{MR}$ we have to take
into account the special situation where both conditions
happen at the same time.

\begin{figure}[htb]
\begin{center}
\includegraphics[width=7cm]{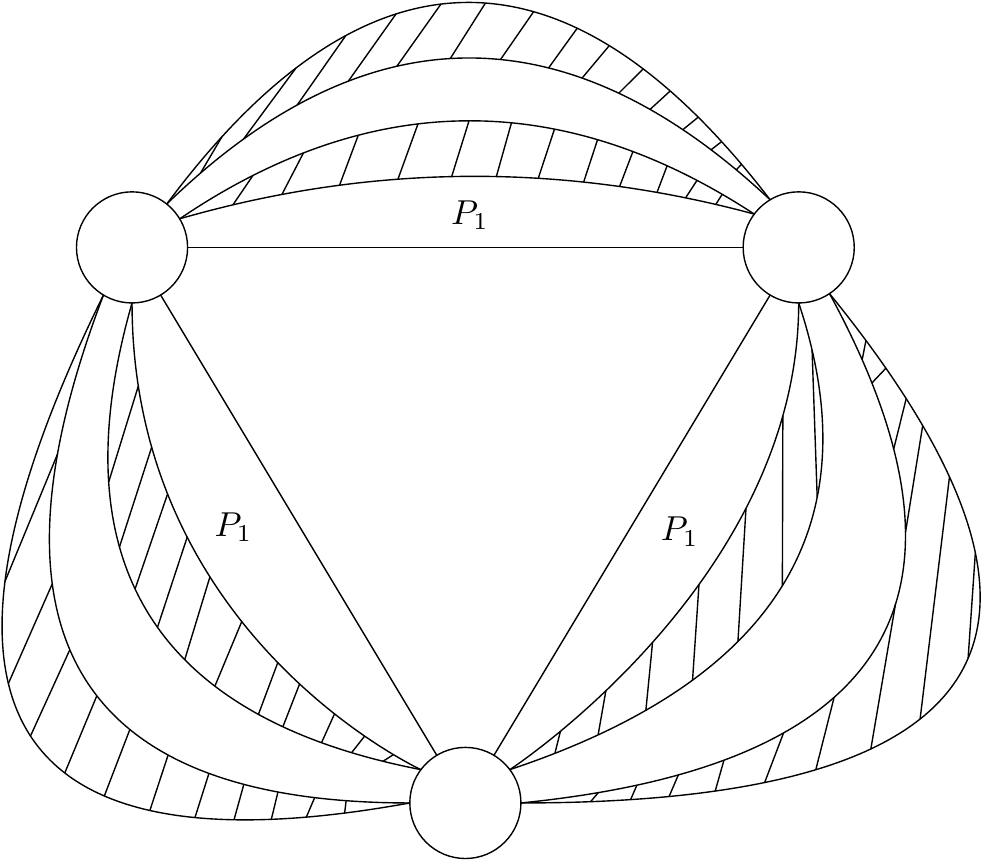}
\end{center}
\caption{ 2-connected series parallel graph rooted in a ring of length 3. The shaded regions plus the edges represent parallel networks of the kind $P^\blacktriangle_1$.}
\label{graphicring}
\end{figure}

Note that this is not a system of equations provided that we
know the values of $P^\blacktriangle_0$, $P^\blacktriangle_1$, $S^\blacktriangle_2$ and $S^\blacktriangle_3$. Finally, following \cite{Chapuy2008}, the generating function $B^\blacktriangle(x,y,u)$ is obtained as
\begin{equation}\label{eq:B_triangles}
B^\blacktriangle (x,y,u) =\frac{1}{2}x^2 y+ B^\blacktriangle_R(x,y,u)+B^\blacktriangle_M(x,y,u)-B^\blacktriangle_{MR}(x,y,u).
\end{equation}
In the last step we just compute the generating function $G^\blacktriangle(x,y,u)$ as the set of its connected components,
encoded as the exponential of $C^\blacktriangle(x,y,u)$, which at the same time can be obtained from the
decomposition into 2-connected components, encoded as $B^\blacktriangle (x,y,u)$, by a standard integration.
This determines the generating function $G^\blacktriangle(x,y,u)$ of SP graphs where $x$ counts vertices,
$y$ counts edges and $u$ counts triangles.

\subsubsection{Number of triangles}\label{sec:tri}
Now that we have the generating function $G^\blacktriangle(x,y,u)$,
by means of the Quasi-Powers Theorem (see~\cite{quasi-powers})
we can show that the number of triangles tends to a normal law,
and obtain its mean and variance. We will not consider
the number of edges any more, so we can assume that $y=1$.
In all this section all generating functions are evaluated at this point, and hence, we only use variable $x$ and $u$.

The first lemma gives the singularity type of the networks:

\begin{lemma}\label{lem:S3}
The generating function $S^\blacktriangle_3(x,u)$ of networks
where the distance between the poles is greater than 2
satisfies
$$
S^\blacktriangle_3(x,u)=g_3(x,u)-h_3(x,u)\sqrt{1-{x\over R^\blacktriangle(u)}},
$$
for functions $g_3(x,u)$ and $h_3(x,u)$ analytic in a neighbourhood of the point $(x,u)=(R,1)$, $R=R^\blacktriangle(1) \approx 0.12800$ , and where $R^\blacktriangle(u)$ is the
singularity curve of $S^\blacktriangle_3(x,u)$.
\end{lemma}

\begin{proof}
We will use the techniques shown in~\cite{DrmotaBook}.
In particular, we will use~\cite[Theorem 2]{DrFuKaKrRu11}, which is a consequence of~\cite[Theorem 2.33]{DrmotaBook}.
First, we need to adapt the equations so that they satisfy the hypothesis of~\cite[Theorem 2]{DrFuKaKrRu11}.
The new equations are:
$$
S^\blacktriangle_2 = x \exp(2uS^\blacktriangle_2+2S^\blacktriangle_3) = F_1(x,S^\blacktriangle_2,S^\blacktriangle_3,u)
$$
$$
S^\blacktriangle_3 = x\left(\exp_{\geq 2}(S^\blacktriangle_2+S^\blacktriangle_3)
(\exp_{\geq 1}(S^\blacktriangle_2+S^\blacktriangle_3)+\exp(uS^\blacktriangle_2+S^\blacktriangle_3))+\right.
$$
$$
\left.
\exp(uS^\blacktriangle_2+S^\blacktriangle_3)\exp_{\geq 1}(S^\blacktriangle_2+S^\blacktriangle_3)\right)=F_2(x,S^\blacktriangle_2,S^\blacktriangle_3,u)
$$
Since the functions $F_1, F_2$ are analytic in the complex plane, they satisfy the hypothesis of~\cite[Theorem 2]{DrFuKaKrRu11}.
Moreover, in~\cite{BoGiKaNo07} the authors show that for $u=1$ the system has a unique solution, for which $x=R \approx 0.12800$.
Since the system is aperiodic, there is a unique singularity,
which implies the existence of a square-root expansion around $(x=R,u=1)$, which in particular implies the statement.
\end{proof}

By means of Equation~\eqref{eq:D}, all different network classes can be expressed in terms of both $S^\blacktriangle_3(x,u)$ and $S^\blacktriangle_2(x,u)$.
Hence, all network classes have a similar expression.
This observation makes the following lemma an straightforward result:

\begin{lemma}
The generating function $B^\blacktriangle(x,u)$ of 2-connected SP graphs where $x$ marks vertices and $u$ marks triangles satisfies
\begin{equation}\label{eq:B_tri}
B^\blacktriangle(x,u)=g_B(x,u)-h_B(x,u)\sqrt{1-{x\over R^\blacktriangle(u)}},
\end{equation}
where $g_B$ and $h_B$ are analytic in a neighbourhood of the point $(x,u)=(R,1)$, $R=R^\blacktriangle(1) \approx 0.12800$, and where $R^\blacktriangle(u)$ is the function described in Lemma~\ref{lem:S3}.
\end{lemma}

%\begin{proof}
%$B^\blacktriangle(x,u)$ can be expressed explicitly in terms of $S^\blacktriangle_3(x,u)$. Hence the singular behaviour is of the same type.
%\end{proof}

As described in \cite{3-con}, the dominant singularity of both $C(x,u)$ and $G(x,u)$ arises from a branch point of the equation defining $C^{\bullet}(x,u)$ in terms of  $B^{\circ}(x,u)$. We write $\tau^{\blacktriangle}(u)$ the solution to the equation $\tau^{\blacktriangle}(u){B^{\circ}}^{'}(\tau^{\blacktriangle}(u),u)=1$. The singularity of $C(x,u)$ (and also $G(x,u)$) is located at $\rho^{\blacktriangle}(u)=\tau^{\blacktriangle}(u) \exp(-B^{\circ}(\tau^{\blacktriangle}(u),u))$.
Note that, since both $C(x,u)$ and $G(x,u)$ are aperiodic, the singularity is
unique.

Next step is to deduce from the previous lemmas the limiting distribution for the number of triangles. We already know that in the connected level this random variable follows a normal variable. In the next lemma we particularize the result in the case of 2-connected graphs in the family:

\begin{theorem}\label{th:number2}
The number of triangles $W_n^{\blacktriangle}$ of a uniformly at random 2-connected SP graph with $n$ vertices
is asymptotically normal distributed, with
$$
\mathbb{E}[W_n^{\blacktriangle}] = \mu_{\blacktriangle,2} n(1+o(1)),\qquad \mathbb{V}\mathrm{ar}[W_n^{\blacktriangle}] = \sigma_{\blacktriangle,2}^2 n(1+o(1)),
$$
where $\mu_{\blacktriangle,2} \approx 0.45242$ and $\sigma_{\blacktriangle,2}^{2} \approx 0.45997$.
\end{theorem}

\begin{proof} To get the constant in the expectation and variance, we compute both ${R^\blacktriangle}'(u)$ and ${R^\blacktriangle}''(u)$ by means of the equations for networks. As both parameters
$$\mu_{\blacktriangle,2}=-\frac{{R^\blacktriangle}'(1)}{{R^\blacktriangle}(1)},\,\, \sigma_{\blacktriangle,2}^{2}=-\frac{{R^\blacktriangle}''(1)}{{R^\blacktriangle}(1)}-\frac{{R^\blacktriangle}'(1)}{{R^\blacktriangle}(1)}+\left(\frac{{R^\blacktriangle}'(1)}{{R^\blacktriangle}(1)}\right)^2.$$
are strictly greater than $0$, we can apply the Quasi-Powers Theorem over the expression in Equation~\eqref{eq:B_tri}, and the result holds straightforward.
\end{proof}

Finally, we are able to compute the number of triangles in a uniformly at random SP graph of size $n$.

\begin{theorem}\label{th:number}
The number of triangles $X_n^{\blacktriangle}$ of a SP graph with $n$ vertices
is asymptotically normal, with
$$
\mathbb{E}[X_n^{\blacktriangle}] = \mu_{\blacktriangle} n+O(1),\qquad \mathbb{V}\mathrm{ar}[X_n^{\blacktriangle}] = \sigma_{\blacktriangle}^2 n+O(1),
$$
where $\mu_{\blacktriangle} \approx 0.39481$ and $\sigma_{\blacktriangle}^{2} \approx 0.41450$.
\end{theorem}

\begin{proof}
The normality of the random variable is assured by the fact that SP graphs are subcritical, which implies that we can apply
\cite[Theorem 2.23]{DrmotaBook}
to the decomposition of a connected graph into 2-connected blocks:
$$
C^{\bullet} = F (x, C^{\bullet}, u) = x \exp (B^\circ(C^\bullet,u)),
$$
and we obtain $\mathbb{E}[X_n^{\blacktriangle}]$ and
$\mathbb{V}\mathrm{ar}[X_n^{\blacktriangle}]$ from the explicit expression
of $F$ in terms of $B^\circ$.
\end{proof}

Observe the same limiting distribution holds for the number of triangles in a uniformly at random (general) SP graph on $n$ vertices.

One may compare these values with the expected number of pending triangles in a random SP graph, computed in \cite{3-con} as approximately
$2.2313\cdot 10 ^ {-3}n$. As expected, the number of appearances of a triangle in a random SP graph is much smaller than the
number of occurrences of the triangle in a random SP graph.

\subsubsection{Enumeration of triangle-free series-parallel graphs}\label{sec:enum}

If we write $u=0$ in the equations of the previous section we get the generating function of triangle-free SP graphs.
In this subsection we provide the asymptotic analysis of such family, which is interesting by itself.

In all this section we use the equations in the introduction of Section \ref{ss.triangles-SP} with the value $u=0$.
In order to emphasize that we are considering triangle-free families, we use the superindex $\triangle$ instead of $\blacktriangle$.
We start studying the singular behaviour for networks.

\begin{lemma}
Fix $y$ in an small neighbourhood of 1. The generating function $S^{\triangle}_3(x,y)$ of triangle-free series networks
where the poles are at a distance greater than 2 has a positive singularity $R^{\triangle}(y)$, and the following singular expansion in a dented domain $\Delta$ at a certain $x=R_{\triangle}(y)$:
$$
S^{\triangle}_3(x,y) = a^{\triangle}_0(y) + a^{\triangle}_1(y)X+a^{\triangle}_2(y)X^2+a^{\triangle}_3(y)X^3+O(X^4),
$$
where $X = \sqrt{1-x/R_{\triangle}(y)}$. In particular, $R_{\triangle}(1)\approx 0.19635$.
\end{lemma}
\begin{proof}
First, note that if we assign $u=0$ in the equations defining the networks in Equation~\eqref{eq:D}, then we can express $S^{\triangle}_3$ as the solution of the following single implicit equation:
\begin{equation}\label{eq:S3}
S^{\triangle}_3=x \left( \left( \exp_{\geq 2}(x{y}^{2}{{\rm e}^{2{S^{\triangle}_3}}}+{S^{\triangle}_3})\right)  \left({\exp_{\geq 1}({x
{y}^{2}{{\rm e}^{2\,{\it S^{\triangle}_3}}}+{\it S^{\triangle}_3}})}+y{{\rm e}^{{\it S^{\triangle}_3}}}
 \right) +y{{\rm e}^{{\it S^{\triangle}_3}}} {\exp_{\geq 1}({x{y}^{2}{{\rm e}^{2\,
{\it S^{\triangle}_3}}}+{\it S^{\triangle}_3}})} \right).
\end{equation}
We write the right hand side of Equation \eqref{eq:S3} as $G(S^{\triangle}_3,x,y)$. Then, for every choice of $y$ in a neighbourhood of $1$,  we need to check that $S^{\triangle}_3$ satisfies a so-called \emph{smooth implicit-function scheme} (see the work of Meir and Moon \cite{MeirMoon89}, se also \cite[Section VII. 4.1.]{fs05}) of the form $S^{\triangle}_3=G(S^{\triangle}_3,x,y)$. In this context, if $G$ verifies some analytic conditions, then the solution $S^{\triangle}_3(x,y_0)$ of the equation admits an square root expansion in a domain dented at its singularity. We now check the conditions:
\begin{enumerate}
\item $G(U,x,y)$ must be analytic in a given complex region. In our case it is an entire function.
\item The coefficients $g_{m,n}(y)$ of the Taylor expansion of $G(U,x,y)$ with respect to $U$ and $x$ must be non-negative, as it is the case. Moreover  $g_{0,0}(y)=0$, and $g_{0,1}(y) = 0 \ne 1$.
\item $g_{m,n}(y)$ must be positive for some $m$ and some $n\geq 2$.
Since $g_{1,2}(y)=2y$, this holds for any $y$ in an small neighbourhood of $1$.
\item The singularity must be unique, which is true since the generating function
is aperiodic.
\item Finally, for each choice of $y$ in an small neighbourhood of $1$,  we need the existence of a solution $R_{\triangle}(y)$ and $a_0^{\triangle}(y)$ satisfying the characteristic system
\begin{equation}\label{eq:sing-triagn}
a^{\triangle}(y)= G(a_0^{\triangle}(y),R_{\triangle}(y),y),\,\, 1 = G_U(a_0^{\triangle}(y),R_{\triangle}(y),y).
\end{equation}
Direct computations for $y=1$ gives that such system of equations has a valid solution at $a^{\triangle}(1)\approx 0.15545 $ and $R_{\triangle} \approx 0.19635$. Finally, this statement is also true in an small neighbourhood of $y=1$ by the fact that both equations in system
\eqref{eq:sing-triagn}.
\end{enumerate}

In conclusion, the implicit-function scheme $S^{\triangle}_3=G(x,S^{\triangle}_3,y)$ is smooth for $y=1$, by continuity of the equations it is also smooth for $y$ in an small neighbourhood of $y$. Hence for each choice of $y$ in an small neighbourhood of $1$,  $S^{\triangle}_3$ admits a square-root expansion in a domain dented at $R_{\triangle}(y)$, as we wanted to show.
\end{proof}

Once we have the singularity behaviour of $S^{\triangle}_3(x,y)$ fore $y$ close enough to $1$, we can compute the coefficients of its singular expansion at a given value of $y=1$. This computation is enclosed in the following lemma.

\begin{lemma}
We have that the coefficients on the singular expansion of $S^{\triangle}_3(x,1)$ in a domain dented at $x=R_{\triangle}(1) \approx 0.19635$ are equal to:
\begin{eqnarray}\label{eq:sing_S3_triang}
a^{\triangle}_0(1) = a^{\triangle}_0 \approx 0.15545,& a^{\triangle}_1(1) = a^{\triangle}_1 \approx -0.34792,\\
a^{\triangle}_2(1) = a^{\triangle}_2 \approx 0.27799,& a^{\triangle}_3(1) = a^{\triangle}_3 \approx -0.16276.\nonumber
\end{eqnarray}
In particular, $a^{\triangle}_i(y)\neq 0$ for $i\in \{0,\ldots,3\}$ and $y$ in an small neighbourhood of $1$.
\end{lemma}
\begin{proof}
We just apply undeterminate coefficients over Equation~\eqref{eq:S3}. By continuity of the functions $a^{\triangle}_i(y)$, the final statement holds as well.
\end{proof}
By using this singular expansion for $S^{\triangle}_3(x,1)$ we can obtain the corresponding coefficients of the singular expansion of the rest of the networks counting formulas:
\begin{lemma}
The generating functions $P^{\triangle}_0$, $P^{\triangle}_1$, $S^{\triangle}_2$ and $D^{\triangle}$ have the following singular expansions in a domain dented at $x=R_{\triangle}(y)$:
\begin{eqnarray}
P^{\triangle}_0 &=& p^{\triangle}_0(y_0)+p^{\triangle}_1(y_0) X +p^{\triangle}_2(y_0)X^2+p^{\triangle}_3(y_0)X^3+O(X^4)\\
P^{\triangle}_1 &=& q^{\triangle}_0(y_0)+q^{\triangle}_1(y_0) X +q^{\triangle}_2(y_0)X^2+q^{\triangle}_3(y_0)X^3+O(X^4) \nonumber \\
S^{\triangle}_2 &=& s^{\triangle}_0(y_0)+s^{\triangle}_1(y_0) X +s^{\triangle}_2(y_0)X^2+s^{\triangle}_3(y_0)X^3+O(X^4) \nonumber \\
D^{\triangle}   &=& d^{\triangle}_0(y_0)+d^{\triangle}_1(y_0) X +d^{\triangle}_2(y_0)X^2+d^{\triangle}_3(y_0)X^3+O(X^4), \nonumber
\end{eqnarray}
where $X = \sqrt{1-x/R_{\triangle}(y)}$. In particular, when $y=1$ and $R_{\triangle}(1)\approx 0.19635$ we have that
$$p^{\triangle}_0(1)\approx 0.10374,\,\,   p^{\triangle}_1(1)\approx -0.28169,\,\,  p^{\triangle}_2(1)\approx 0.33606,\,\,       p^{\triangle}_{3}(1) \approx -0.31761, $$
$$q^{\triangle}_0(1)\approx 1.16818, \,\,  q^{\triangle}_1(1)\approx -0.40643,\,\,  q^{\triangle}_2(1)\approx 0.39544,\,\,   q^{\triangle}_3(1)\approx -0.31132,$$
$$d^{\triangle}_0(1)\approx 1.69532,\,\,   d^{\triangle}_1(1)\approx -1.22249,\,\,  d^{\triangle}_2(1)\approx 0.95538,\,\,   d^{\triangle}_3(1)\approx -0.81117,$$
$$s^{\triangle}_0(1)\approx 0.26795, \,\,  s^{\triangle}_1(1)\approx -0.18645,\,\,  s^{\triangle}_2(1)\approx -0.05411,\,\,  s^{\triangle}_3(1)\approx -0.01948,$$
Additionally, all the terms are different to 0 when $y$ belongs to an small neighbourhood of $1$.
\end{lemma}
\begin{proof}
Observe that $P^{\triangle}_0$, $P^{\triangle}_1$, $S^{\triangle}_2$ and $D^{\triangle}$ can be expressed explicitly in terms of $S^{\triangle}_3$.
Hence we can use the coefficients of the expansion of $S^{\triangle}_3$ obtained in \eqref{eq:sing_S3_triang} to compute the expansion for the all other functions.
The last statement follows from continuity and the fact that the computations give coefficients different from 0.
\end{proof}
We can go now directly to get the singular expansions for $B^{\triangle}(x,y)$:
\begin{theorem}
The generating function $B^{\triangle}(x,y)$ of triangle-free SP networks has the following singular
expansion in a domain dented at $R_{\triangle}$ of the form
$$
B^{\triangle} = b^{\triangle}_0(y) + b^{\triangle}_2(y)X^2+b^{\triangle}_3(y)X^3+O(X^4),
$$
where $X = \sqrt{1-x/R_{\triangle}(y)}$. In particular, when $y=1$ and $R_{\triangle}(1)\approx 0.19635$ we have that
$$
b^{\triangle}_0(1) \approx 0.01964 , \qquad b^{\triangle}_2(1) \approx -0.04123,
\qquad b^{\triangle}_3(1) \approx  0.00359.
$$
Moreover, $b^{\triangle}_0(y)$, $b^{\triangle}_2(y)$ and $b^{\triangle}_3(y)$
are different to 0 for $y$ close enough to 1, and the term $b^{\triangle}_1(y)$ of $X$ is 0 for
any $y$ close enough to 1.
\end{theorem}
\begin{proof}
Replacing $P^{\triangle}_0$, $P^{\triangle}_1$, $S^{\triangle}_2$, $S^{\triangle}_3$ with their
singular expansion in the equation \eqref{eq:B_triangles} gives directly square-root expansions for $B^{\triangle}$.
Note that $B^{\triangle}$ can also be obtained from $P^{\triangle}_1$ by the equation
$$
2y{\partial B^{\triangle}\over \partial y} = x^2 P^{\triangle}_1.
$$
This is true because by our encoding the only networks which contains the root edge are the ones considered in $P^{\triangle}_1$.
This implies that the singular expansion of $B^{\triangle}$ must start at $X^3$, so that after differentiating it we get the singular expansion of $P^{\triangle}_1$.
\end{proof}
We can now apply the Transfer Theorem for singularity analysis \cite{flajolet1990singularity} in order to get the first asymptotic counting formulas:

\begin{theorem}\label{thm:main2} The number $b^{\triangle}_n$ of 2-connected triangle-free SP graphs with $n$ vertices is asymptotically equal to
$$b^{\triangle}\cdot n^{-5/2}\cdot R_{\triangle}^{-n}\cdot n!\,\,(1+o(1)),$$
where $b^{\triangle}\approx 0.00152$ and $R_{\triangle}^{-1} \approx 5.09289$.
\end{theorem}
\begin{proof}
Applying the Transfer Theorem to the singular expansion.
\end{proof}

Now we can move to the connected level. In this case, the solution of the equation $\tau {B^{\bullet}} '(\tau)=1$ is located at $\tau = 0.19631$, and hence the singularity of $C^{\bullet}(x)$ is located at $\rho=0.19403$ . We can then state the final enumerative theorem in this subsection:

\begin{theorem}
The number of connected and general triangle-free SP graphs with $n$ vertices ($c^{\triangle}_n$ and $g^{\triangle}_n$, respectively) is asymptotically equal to
$$c^{\triangle}\cdot n^{-5/2}\cdot \rho_{\triangle}^{-n}\cdot n!\, (1+o(1)) ,\,\,\,\,\, g^{\triangle}\cdot n^{-5/2}\cdot \rho_{\triangle}^{-n}  \cdot n!\, (1+o(1)),$$
where $c^{\triangle}\approx 0.00473$, $g^{\triangle}\approx 0.00563$  and $\rho_{\triangle}^{-1} \approx 6.28155$.
\end{theorem}

\begin{proof}
This is an straightforward computation. Due to the subcritical scheme, the singularity of both $C^{\triangle}(x,1)$ and $G^{\triangle}(x,1)$ arise from a branch point. The solution to the equation $x{B^{\circ}}'(x)=1$ is given by $\tau_\triangle=0.19629$. Such value gives that $C(x)$ ceases to be analytic at $x=\rho_\triangle=0.15920$ .We apply then the Transfer Theorem to the resulting singular expansion, joint with the expressions of the coefficients of the singular expansions that were obtained in \cite[Proposition 3.10.]{3-con}.
\end{proof}

As a direct consequence of these computations, the probability that a uniformly at random triangle-free SP graph of size $n$ is connected is equal to $\exp(-C_0)\approx 0.83962$ (see \cite[Theorem 4.6.]{3-con}).

These enumerative results complement previous ones concerning SP graphs with certain obstructions. In Table~\ref{table:constants}, the constant growth for (connected) SP graphs, triangle-free SP graphs and bipartite SP graphs is shown. The constant for the full family was obtained in \cite{BoGiKaNo07}, while the asymptotic enumeration for bipartite SP graphs can be found in \cite{WelReqRue2015}.

\begin{table}[htb]
\begin{center}
\begin{tabular}{c| c}
  Family & Constant growth\\\hline
  Series-Parallel & $9.07359$\\
  Triangle-free Series-Parallel & $6.28155$\\
  Bipartite Series-Parallel & $5.30386$\\
  \end{tabular}\\
\caption{Constants growth for series-parallel graphs, and for subfamilies (triangle-free and bipartite). }\label{table:constants}
\end{center}
\end{table}

It is interesting to observe that the asymptotics for triangle-free graphs and bipartite graphs is different. This fact contrasts with the picture that emerges in the general graph setting: as it was proven by Erd\H{o}s, Kleitman and Rothschild in \cite{ErKleRot76}, the number of triangle-free graphs with $n$ vertices is asymptotically equal to the number of bipartite graphs with $n$ vertices.

\subsection{4-cycles}\label{ss.4-cycles-sp}

For the sake of conciseness and in order to show a new set of equations we analyze the statistics of $4$-cycles $C_4$ in 2-connected SP graphs. We proceed as we did in the previous section: we get first the equations defining networks, and then we build the counting formulas of 2-connected SP graphs. We do not deal with the connected and general setting,
because we are in the subcritical case and the procedure
will be very similar to the case of triangles.

The combinatorial ideas to get the generating functions for networks (encoding now the number of cycles of length 4) are similar to the ones used before. We denote by $S_2^\blacksquare$, $S_3^\blacksquare$ and $S_\infty^\blacksquare$ series networks where the poles are at distance 2, 3 and more than 3, respectively. Observe that the first two networks could contribute to the creation of $4$-cycles (by means of parallel operations), while the term $S_\infty^\blacksquare$ cannot.
Similarly, we define $P_1^\blacksquare$, $P_2^\blacksquare$ and $P_\infty^\blacksquare$ parallel networks where the distance of the poles is equal to $1$, 2 or more than 2. In particular the single edge is encoded in $P_1$. The total counting formula for networks is encoded by $D^\blacksquare$.

Note that there is a difference with respect
to triangles. A series network where the poles
are at distance 2 has a unique path of length
2 between the poles. This is not true for the case
of paths of length 3: both $S_2^\blacksquare$ and $S_3^\blacksquare$

may have an arbitrary number of paths of length 3,
and each of those will form a 4-cycle if
we put it in parallel with an edge. This is why
we need two additional functions,
$\overline{S_2^\blacksquare}$ and $\overline{S_3^\blacksquare}$,
that count series networks that will be put in
parallel with an edge, and therefore each path
of length three will contribute with a new
4-cycle. In other words, in $\overline{S_2^\blacksquare}$
and $\overline{S_3^\blacksquare}$ the variable $u$ counts
both 4-cycles and paths of length 3 between
the poles.

In order to count paths of length 3 we need to
note that some of them come from paths of length
2 in the parallel networks that we put in series.
Therefore, we will denote
as $\overline{P_1^\blacksquare}$ and $\overline{P_2^\blacksquare}$
the parallel networks where $u$ counts both
4-cycles and paths of length 2 between the poles.
These paths in turn come from series networks,
so we use again the property that a series
network whose poles are at distance 2 have
a unique path of length 2 between the poles.
This gives the following equations:

%We use again the variable $u$ to encode appearances of the $C_4$. Then we have the following equations:

\begin{align}\label{eq:C4-networks}
D^\blacksquare&=\, S_2^\blacksquare+S_3^\blacksquare+S_\infty^\blacksquare+
P_1^\blacksquare+P_2^\blacksquare+P_\infty^\blacksquare,\\
S_2^\blacksquare &=\, x (P_1^\blacksquare)^2,\nonumber\\
\overline{S_2^\blacksquare} &=\, x\left(
\overline{P_1^\blacksquare}\right)^2,\nonumber\\
S_3^\blacksquare &=\, x(P_1^\blacksquare(P_2^\blacksquare+S_2^\blacksquare) + P_2^\blacksquare P_1^\blacksquare),\nonumber\\
\overline{S_3^\blacksquare} &=\,
x(P_1^\blacksquare(\overline{P_2^\blacksquare}+uS_2^\blacksquare) + \overline{P_2^\blacksquare} P_1^\blacksquare),\nonumber\\
S_\infty^\blacksquare &=\, x (P_1^\blacksquare(P_\infty^\blacksquare + S_3^\blacksquare + S_\infty^\blacksquare) + P_2^\blacksquare (S_2^\blacksquare + S_3^\blacksquare + S_\infty^\blacksquare +P_2^\blacksquare + P_\infty^\blacksquare) + P_\infty^\blacksquare D^\blacksquare),\nonumber\\
P_1^\blacksquare &=\, y\left(\exp(\overline{S_3^\blacksquare}+S_\infty^\blacksquare)
\sum_{k\geq 0}u^{\binom{k}{2}}{\left(\overline{S_2^\blacksquare}\right)^k \over k!}\right),\nonumber\\
\overline{P_1^\blacksquare} &=\, y\left(\exp(\overline{S_3^\blacksquare}+S_\infty^\blacksquare)
\sum_{k\geq 0}u^{\binom{k+1}{2}}{\left(\overline{S_2^\blacksquare}\right)^k \over k!}\right),\nonumber\\
P_2^\blacksquare &=\, S_2^\blacksquare \exp_{\geq 1} (S_3^\blacksquare+S_\infty^\blacksquare) + \exp (S_3^\blacksquare+S_\infty^\blacksquare)\sum_{k \geq 2}u^{\binom{k}{2}}{(S_2^\blacksquare)^k\over k!} ,\nonumber\\
\overline{P_2^\blacksquare} &=\, uS_2^\blacksquare \exp_{\geq 1} (S_3^\blacksquare+S_\infty^\blacksquare) + \exp (S_3^\blacksquare+S_\infty^\blacksquare)\sum_{k \geq 2}u^{\binom{k+1}{2}}{(S_2^\blacksquare)^k\over k!} ,\nonumber\\
P_\infty^\blacksquare &=\, \exp_{\geq 2} (S_3^\blacksquare+S_\infty^\blacksquare).\nonumber\\ \nonumber
\end{align}

The equations for series networks are obtained by fixing the network type which is incident with the $0$-pole (which must be of parallel type), and the network incident with the $\infty$-pole.
In particular, the indices must sum the corresponding index in $S$.
The equations for parallel networks are more involved: in this case sets of networks of type $S_2^\blacksquare$ can create a quadratic number of copies of  $C_4$, hence the infinite sums with quadratic exponents in $u$.

Starting from these equations, we can go to deduce counting formulas for 2-connected objects. In order to apply the dissymmetry theorem we need to obtain
the corresponding $B_R^\blacksquare$, $B_M^\blacksquare$ and $B_{RM}^\blacksquare$ as follows:
$$
B_R^\blacksquare = \text{Cyc} (x(P_1^\blacksquare+P_2^\blacksquare+P_\infty^\blacksquare))
+{x^3\over 6}\left(\left(\overline{P_1^\blacksquare}\right)^3-(P_1^\blacksquare)^3+
3(P_1^\blacksquare)^2\overline{P_2^\blacksquare}-
3\left(P_1^\blacksquare\right)^2P_2^\blacksquare\right)+
(u-1){(xP_1^\blacksquare)^4\over 8}
$$
$$
B_M^\blacksquare = {x^2\over 2}\left(
P_1^\blacksquare+P_2^\blacksquare+P_\infty^\blacksquare -
\left(y+y(\overline{S_2^\blacksquare}+\overline{S_3^\blacksquare}
+S_\infty^\blacksquare)
+{u(S_2^\blacksquare)^2 \over 2}+S_2^\blacksquare(S_3^\blacksquare+S_\infty^\blacksquare)
+{(S_3^\blacksquare+S_\infty^\blacksquare)^2 \over 2}
\right)\right)
$$
$$
B_{RM}^\blacksquare = {x^2\over 2}(\overline{S_2^\blacksquare}
(\overline{P_1^\blacksquare}-y)+
S_2 (\overline{P_2^\blacksquare} + P_\infty^\blacksquare)
+(\overline{S_3^\blacksquare}+S_\infty^\blacksquare)(P_1^\blacksquare-y)+
(S_3^\blacksquare+S_\infty^\blacksquare)(P_2^\blacksquare+P_\infty^\blacksquare))
$$
$$
B = {1 \over 2}x^2 y + B_R^\blacksquare+B_M^\blacksquare-B_{RM}^\blacksquare,
$$
where $B_R^\blacksquare$ represents 2-connected series-parallel
graphs rooted at a ring, $B_M^\blacksquare$
represents 2-connected series-parallel
graphs rooted at a multiedge, and
$B_{RM}^\blacksquare$ represents 2-connected series-parallel
graphs rooted at a ring and a multiedge that are
adjacent at the decomposition tree.
In the case of $B_R^\blacksquare$ we have to deal with several
special cases, since if the length of the ring
is 3 or 4, then 4-cycles might appear. If the length
is 4, a single cycle appears. If the length
is 3, many cycles might appear if we replace at least
two edges of the ring with a parallel network of
the kind $P_1^\blacksquare$: in particular, any path of length
2 in the other parallel network will produce a 4-cycle.
In the case of $B_M^\blacksquare$, the parallel networks already
count all the 4-cycles, but since the number
of edges must be at least three, we have to remove
the cases with one and two series networks or edges in
parallel.
In the case of $B_{RM}^\blacksquare$ we have to consider the special
cases depending on the length of the ring and whether
there is an edge between the poles of the multiedge.
If the length of the ring is three, then any path
of length two in the parallel network will produce
a 4-cycle, whereas if there is an edge between the
poles of the multiedge, then any path of length
three in the series network will produce a 4-cycle,
as it is shown in Figure~\ref{graphicringmulti}.

\begin{figure}[htb]
\begin{center}
\includegraphics[width=7cm]{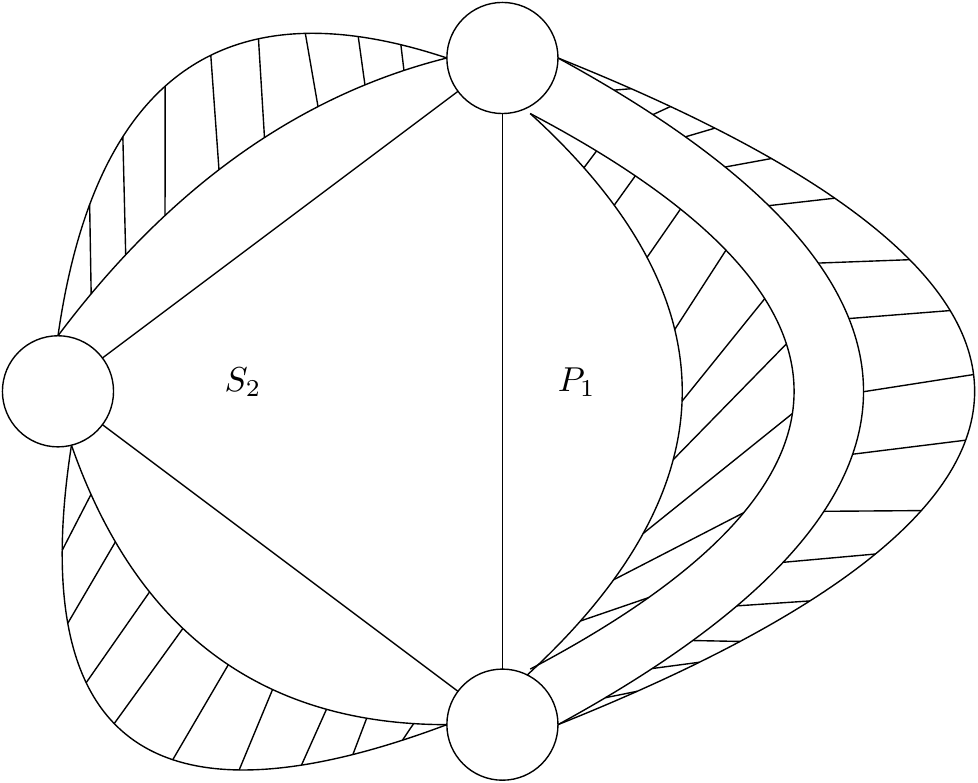}
\end{center}
\caption{ 2-connected series parallel graph rooted in a ring of length 3 and a multiedge with an edge between the poles. Each path of length 3 in $S_2^\blacksquare$ and
each path of length 2 in $P_1^\blacksquare$
will produce a new 4-cycle.}
\label{graphicringmulti}
\end{figure}

Using these expression
and setting $y=1$, $u=0$ we obtain the
radius of convergence and the singularity analysis
of $B^\blacksquare(x,1,0)$, which,
by means of the Transfer Theorem,
gives the asymptotic enumeration
of 2-connected SP-graphs without 4-cycles.
\begin{theorem}
The number of 2-connected quadrangle-free SP graphs with $n$ vertices ($b^{\square}_n$) is asymptotically equal to
$$b^{\square}\cdot n^{-5/2}\cdot R_{\square}^{-n}  \cdot n!\, (1+o(1)),$$
where $b^{\square}\approx 0.00145$
and $R_{\square}^{-1} \approx 5.13738$.
\end{theorem}

The proof of the next result is analogous to the proof of
triangle-free SP graphs.

\begin{theorem}
The number of connected and general quadrangle-free SP graphs with $n$ vertices ($c^{\square}_n$ and $g^{\square}_n$, respectively) is asymptotically equal to
$$c^{\square}\cdot n^{-5/2}\cdot \rho_{\square}^{-n}\cdot n!\, (1+o(1)) ,\,\,\,\,\, g^{\square}\cdot n^{-5/2}\cdot \rho_{\square}^{-n}  \cdot n!\, (1+o(1)),$$
where $c^{\square}\approx 0.00233$, $g^{\square}\approx 0.00276$  and $\rho_{\square}^{-1} \approx 6.41498$.
\end{theorem}

We use Remark~\ref{rem:system} to obtain the following result about 4-cycles.
It is a modification of \cite[Theorem 2.35]{DrmotaBook}, which provides
a way to compute the expectation and variance of generating functions
that satisfy the following system of equations:
\begin{equation}\label{eq:system1}
\boldsymbol{y} = \boldsymbol{F}(x,\boldsymbol{y}, \boldsymbol{u}),
\end{equation}
\begin{equation}\label{eq:system2}
0 =
\det(\boldsymbol{I}-\boldsymbol{F}_{\boldsymbol{y}}(x,\boldsymbol{y},\boldsymbol{u}).
\end{equation}
According to that theorem, the expectation $\boldsymbol\mu$
and the variance $\boldsymbol\Sigma$ of the parameters $\boldsymbol{u}$
can be computed as
$$
\boldsymbol{\mu}=-{x_{0,\boldsymbol{u}}(1)\over x_{0}(1)},
$$
$$
\boldsymbol{\Sigma}=-{x_{0,\boldsymbol{uu}}(1)\over x_{0}(1)}+
\boldsymbol{\mu}\boldsymbol{\mu}^T + \text{diag}(\boldsymbol{\mu}),
$$
where $x = x_0 (\boldsymbol{u})$ and
$\boldsymbol{y} = \boldsymbol{y}_0(\boldsymbol{u})$
are the solutions of the system (\ref{eq:system1}) and
(\ref{eq:system2}). After half an hour of execution time in Maple we get
the following theorem:
\begin{theorem}\label{th:number2}
The number of quadrangles $W_n^{\blacksquare}$ of a uniformly at random 2-connected SP graph with $n$ vertices
is asymptotically Gaussian, with
$$
\mathbb{E}[W_n^{\blacksquare}] = \mu_{\blacksquare,2} n(1+o(1)),\qquad \mathbb{V}\mathrm{ar}[W_n^{\blacksquare}] = \sigma_{\blacksquare,2}^2 n(1+o(1)),
$$
where $\mu_{\blacksquare,2} \approx 0.51235$ and $\sigma_{\blacksquare,2}^{2} \approx 0.25418$.
\end{theorem}

\subsection{Girth}\label{ss.girth}

Now we can generalize the previous results to obtain the generating function
of SP graphs with girth at least $k$, for $k\geq 4$.
We need new notation for the series and parallel networks.
In particular, in order to express the generating function of
networks with girth $k$ we define $S_i$, for $2\leq i \leq k-2$
as the generating function of series networks with girth
$\geq k$ and where the distance between the poles is exactly $i$.
The generating function of series networks with girth $\geq k$
and distance to the poles $\geq k-1$ is expressed as $S_\infty$.
Analogously, we define $P_i$, for $1\leq i \leq k-3$ and $P_\infty$
as the generating function of parallel networks with the same
condition on the distance between the poles.
In this case the generating functions satisfy the following
system of equations:
$$
P_i=\left\{\begin{tabular}{ccc}
$y\exp (S_\infty)$ & if & $i=1$\\
$S_i \exp_{\geq 1}(\sum_{j\geq k-i} S_j)$ & if & $1<i<k/2$\\
$S_i \exp_{\geq 1}(\sum_{j\geq i+1} S_j)+
\exp_{\geq 2}(S_i) \exp(\sum_{j\geq i+1} S_j)$ & if & $k/2\leq i \leq k-3$\\
$\exp_{\geq 2}(S_{k-2}+S_\infty)$ & if & $i=\infty$
\end{tabular}\right.
$$
$$
S_i=\left\{\begin{tabular}{ccc}
$x\left[\sum_{j=1}^{i-1}P_j(S_{i-j}+P_{i-j})\right]$ & if & $1\leq i \leq k-2$\\
$x\left[\sum_{j\geq 1}P_j\sum_{t\geq k-j-1}(S_{t}+P_{t})\right]$ & if & $i=\infty$
\end{tabular}\right.
$$

Note that for convenience we are considering that $S_1$ exists, with
a value of $S_1=0$.
This equations generalize the corresponding ones for
girth 3. For the case of parallel networks we impose that
the the distance between the poles of the two shortest
series networks is at most $k$. This can be done by distinguishing
two cases: if the distance between the poles is $i < k/2$
then there must be one single series network with distance $i$
between the poles. This implies that all the other series networks
must have distance at least $k-i$ between the poles, because
otherwise there would be a cycle of length less than $k$.
If the distance between the poles is $i \geq k/2$,
then no cycle of length less than $k$ can be produced,
so we just have to be sure that the shortest series network
that we put in parallel is of length $i$.
For the case of series
networks no further constraint is needed, since no new cycle
can be produced.

This gives a way to compute the exponential growth for any
possible girth. Since the computations are involved
and analogous to the ones of girth 4 we do not include
the results.

\section{Concluding remarks} \label{sec:concluding}

In this work we have shown normal limiting distributions for the number of copies of a given graph for subcritical graph classes. From our study several challenging questions might be investigated in the future. The proof of our main theorem does not give a systematic way to compute both the  expectation and the variance of the corresponding random variable (we only get that they are linear in $n$). In Section \ref{sec:comp} have exploited extra information concerning the structure of series-parallel graphs in order to get precise constants, but getting a full numerical analysis seems to be very difficult in general. Nevertheless we can use the Benjamni-Schramm limit given in \cite{stufler2015,schben15} to get the
constant for the mean value (for details see \cite{stufler2015}). However, it seems to be very difficult
to obtain a general procedure for computing the constant for the variance.

%Even computing the expectation of the number of $P_3$ in a random connected series-parallel graph seems to be difficult.

Second, we cannot immediately obtain local limit laws for the number of copies of a given graph.
In our analysis we only provided asymptotic information of our generating functions a neighborhood of
$u=1$ (for $|u|=1$). In order to obtain a local limit theorem we need asymptotic information
for all $u$ with $|u|=1$. This is certainly not our or reach but needs a lot of extra work.

Finally, our techniques do not apply to subgraphs in planar-like families (see \cite{3-con}).
Technically speaking, when analyzing subcritical graph classes we have continuously exploited the assumption that the counting formula for the blocks can be considered to be analytic.
Unfortunately, the picture changes dramatically when dealing with planar graphs, as a critical composition scheme arises (see \cite{Gimenez2009,3-con}).
In this context, very little is known concerning the number of subgraphs in the random planar graph model, even concerning the number of triangles.
The only result we know so far is \cite{ReqRue15}, where the authors exploit the fact that triangles in cubic planar graph do not intersect.
Using this combinatorial fact, they are able to show normality for the number of triangles in cubic planar graphs.
This method does not apply in the general planar setting, as an edge can be incident with many triangles.
So new ideas from different sources are needed to attack this problem.

$$\,\,$$
\paragraph{\textbf{Acknowledgments}:} L.R. and J.R. are grateful to the organizers of the workshop 'Enumerative Combinatorics', held in Oberwolfach on 2--8 March 2014, where this work was initiated. They also thank Marc Noy for fruitful discussions concerning the construction of series-parallel graphs without cycles of length four, and for permanent advice and support.

%%%%%%%%%%%%%%%%%%%%%%%%%%%%%%%%%%%%%%%%%%%%%%%%%%%%%%%%%%%%%%%%%%%%%%%%%%%%%%%%%%%%%%%%%%%%%%%%%%%%%%%%%%%%%%%%%%%%%%%%%%%%%%%%%%%%%%%%%%%%%%%%%%%%
%%%%%%%%%%%%%%%%%%%%%%%%%%%%%%%%%%%%%%%%%%%%%%%%%%%%%%%%%%%%%%%%%%%%%%%%%%%%%%%%%%%%%%%%%%%%%%%%%%%%%%%%%%%%%%%%%%%%%%%%%%%%%%%%%%%%%%%%%%%%%%%%%%%%

\bibliographystyle{abbrv} 	
\bibliography{subgraph}

\begin{thebibliography}{10}

\bibitem{BerLaLe98}
F.~Bergeron, G.~Labelle, and P.~Leroux.
\newblock {\em Combinatorial species and tree-like structures}, volume~67.
\newblock Cambridge University Press, 1998.

\bibitem{Bernasconi2009}
N.~Bernasconi, K.~Panagiotou, and A.~Steger.
\newblock The degree sequence of random graphs from subcritical classes.
\newblock {\em Combinatorics, Probability and Computing}, 18(5):647--681, 2009.

\bibitem{BoGiKaNo07}
M.~Bodirsky, O.~Gim{\'e}nez, M.~Kang, and M.~Noy.
\newblock Enumeration and limit laws for series--parallel graphs.
\newblock {\em European Journal of Combinatorics}, 28(8):2091--2105, 2007.

\bibitem{Chapuy2008}
G.~Chapuy, {\'E}.~Fusy, M.~Kang, and B.~Shoilekova.
\newblock A complete grammar for decomposing a family of graphs into
  3-connected components.
\newblock {\em Electronic Journal of Combinatorics}, 15(1):R148, 2008.

\bibitem{ChyDrmKlaKo08}
F.~Chyzak, M.~Drmota, T.~Klausner, and G.~Kok.
\newblock The distribution of patterns in random trees.
\newblock {\em Comb. Probab. Comput.}, 17(1):21--59, 2008.

\bibitem{drmota97systems}
M.~Drmota.
\newblock Systems of functional equations.
\newblock {\em Random Structures Algorithms}, 10(1-2):103--124, 1997.

\bibitem{DrmotaBook}
M.~Drmota.
\newblock {\em Random trees: an interplay between combinatorics and
  probability}.
\newblock SpringerWienNewYork, 2009.

\bibitem{DrFuKaKrRu11}
M.~Drmota, E.~Fusy, M.~Kang, V.~Kraus, and J.~Ru\'e.
\newblock Asymptotic study of subcritical graph classes.
\newblock {\em SIAM J. Discrete Math.}, 25(4):1615--1651, 2011.

\bibitem{Drmota20102}
M.~Drmota, O.~Gim{\'e}nez, and M.~Noy.
\newblock Vertices of given degree in series-parallel graphs.
\newblock {\em Random Structures Algorithms}, 36(3):273--314, 2010.

\bibitem{Drmota2011}
M.~Drmota, O.~Gim{\'e}nez, and M.~Noy.
\newblock The maximum degree of series-parallel graphs.
\newblock {\em Combinatorics, Probability and Computing}, 20(4):529--570, 2011.

\bibitem{DrmGittMor2015}
M.~Drmota, B.~Gittenberger, and J.~F. Morgenbesser.
\newblock Systems of functional equations and infinite dimensional gaussian
  limits distributions in combinatorial enumeration.
\newblock Submitted, available on-line at
  \texttt{http://www.dmg.tuwien.ac.at/drmota/}.

\bibitem{DrmotaNoyAnalco}
M.~Drmota and M.~Noy.
\newblock Extremal parameters in sub-critical graph classes.
\newblock In {\em ANALCO'13}, pages 1--7, 2013.

\bibitem{ErKleRot76}
P.~Erd\H{o}s, D.~Kleitman, and B.~Rothschild.
\newblock Asymptotic enumeration of kn-free graphs.
\newblock {\em International Colloquium on Combinatorial Theory, Atti dei
  Convegni Lincei, Roma}, 2:19--27, 1976.

\bibitem{ErRe60}
P.~Erd\H{o}s and A.~R\'enyi.
\newblock On the evolution of random graphs.
\newblock 5:17--61, 1960.

\bibitem{flajolet1990singularity}
P.~Flajolet and A.~Odlyzko.
\newblock Singularity analysis of generating functions.
\newblock {\em SIAM Journal on discrete mathematics}, 3(2):216--240, 1990.

\bibitem{fs05}
P.~Flajolet and R.~Sedgewick.
\newblock {\em Analytic combinatorics}.
\newblock Cambridge University Press, 2009.

\bibitem{GaWo08}
Z.~Gao and N.~Wormald.
\newblock Distribution of subgraphs of random regular graphs.
\newblock {\em Random Structures \& Algorithms}, 32(1):38--48, 2008.

\bibitem{GaoWor03}
Z.~Gao and N.~C. Wormald.
\newblock Sharp concentration of the number of submaps in random planar
  triangulations.
\newblock {\em Combinatorica}, 23(3):467--486, 2003.

\bibitem{GaWo04}
Z.~Gao and N.~C. Wormald.
\newblock Asymptotic normality determined by high moments, and submap counts of
  random maps.
\newblock {\em Probab. Theory Related Fields}, 130(3):368--376, 2004.

\bibitem{schben15}
A.~Georgakopoulos and S.~Wagner.
\newblock Limits of subcritical random graphs and random graphs with excluded
  minors.
\newblock Available on-line on \texttt{arXiv:1512.03572}.

\bibitem{Gimenez2009}
O.~Gim{\'e}nez and M.~Noy.
\newblock Asymptotic enumeration and limit laws of planar graphs.
\newblock {\em Journal of the American Mathematical Society}, 22(2):309--329,
  2009.

\bibitem{3-con}
O.~Gim\'enez, M.~Noy, and J.~Ru\'e.
\newblock Graph classes with given 3-connected components: asymptotic
  enumeration and random graphs.
\newblock {\em Random Structures and Algorithms}, 42(4):438--479, 2013.

\bibitem{GoulJack83}
I.~P. Goulden and D.~M. Jackson.
\newblock {\em Combinatorial enumeration}.
\newblock A Wiley-Interscience Publication. John Wiley \& Sons, Inc., New York,
  1983.

\bibitem{quasi-powers}
H.-K. Hwang.
\newblock On convergence rates in the central limit theorems for combinatorial
  structures.
\newblock {\em European Journal of Combinatorics}, 19(3):329--343, 1998.

\bibitem{JaLuRu90}
S.~Janson, T.~{\L}uczak, and A.~Ruci{\'n}ski.
\newblock An exponential bound for the probability of nonexistence of a
  specified subgraph in a random graph.
\newblock In {\em Random graphs '87 ({P}ozna\'n, 1987)}, pages 73--87. Wiley,
  Chichester, 1990.

\bibitem{RanGraphs}
S.~Janson, T.~{\L}uczak, and A.~Ruci{\'n}ski.
\newblock {\em Random graphs}.
\newblock Wiley-Interscience Series in Discrete Mathematics and Optimization.
  Wiley-Interscience, New York, 2000.

\bibitem{KaRu81}
M.~Karo{\'n}ski and A.~Ruci{\'n}ski.
\newblock On the number of strictly balanced subgraphs of a random graph.
\newblock In {\em Graph theory (\L ag\'ow, 1981)}, volume 1018 of {\em Lecture
  Notes in Math.}, pages 79--83. Springer, Berlin, 1983.

\bibitem{KiSuVu07}
J.~H. Kim, B.~Sudakov, and V.~Vu.
\newblock Small subgraphs of random regular graphs.
\newblock {\em Discrete Math.}, 307(15):1961--1967, 2007.

\bibitem{colored}
P.~Lieby, B.~Mckay, J.~Mcleod, and I.~Wanless.
\newblock Subgraphs of random $k$-edge-coloured $k$-regular graphs.
\newblock {\em Combinatorics, Probability and Computing}, 18(4):533--549, 2009.

\bibitem{McStWe05}
C.~McDiarmid, A.~Steger, and D.~J. Welsh.
\newblock Random planar graphs.
\newblock {\em Journal of Combinatorial Theory, Series B}, 93(2):187 -- 205,
  2005.

\bibitem{McKay81}
B.~D. McKay.
\newblock Subgraphs of random graphs with specified degrees.
\newblock In {\em Proceedings of the {T}welfth {S}outheastern {C}onference on
  {C}ombinatorics, {G}raph {T}heory and {C}omputing, {V}ol. {II} ({B}aton
  {R}ouge, {L}a., 1981)}, volume~33, pages 213--223, 1981.

\bibitem{McKayICM}
B.~D. McKay.
\newblock Subgraphs of random graphs with specified degrees.
\newblock In {\em Proceedings of the {I}nternational {C}ongress of
  {M}athematicians. {V}olume {IV}}, pages 2489--2501. Hindustan Book Agency,
  New Delhi, 2010.

\bibitem{Mckaydeg}
B.~D. McKay.
\newblock Subgraphs of dense random graphs with specified degrees.
\newblock {\em Combin. Probab. Comput.}, 20(3):413--433, 2011.

\bibitem{MeirMoon89}
A.~Meir and J.~Moon.
\newblock On an asymptotic method in enumeration.
\newblock {\em Journal of Combinatorial Theory, Series A}, 51(1):77 -- 89,
  1989.

\bibitem{Noy15}
M.~Noy.
\newblock Random planar graphs and beyond.
\newblock To appear at \emph{Proceedings of the International Congress of
  Mathematicians 2014.}

\bibitem{panagiotou2015}
K.~Panagiotou, B.~Stufler, and K.~Weller.
\newblock Scaling limits of random graphs from subcritical classes.
\newblock To appear in \emph{Annals of Probability}. Available on-line at
  \texttt{arXiv:1411.1865}.

\bibitem{ReqRue15}
C.~Requil\'e and J.~Ru\'e.
\newblock Triangles in random cubic 3-connected planar graphs.
\newblock Electronic Notes in Discrete Mathematics 2015.

\bibitem{WelReqRue2015}
C.~Requil\'e, J.~Ru\'e, and K.~Weller.
\newblock The ising model on planar graphs: grammar and applications to
  bipartite enumeration.
\newblock In preparation.

\bibitem{Ruc90}
A.~Ruci{\'n}ski.
\newblock Small subgraphs of random graphs---a survey.
\newblock In {\em Random graphs '87 ({P}ozna\'n, 1987)}, pages 283--303. Wiley,
  Chichester, 1990.

\bibitem{stufler2015}
B.~Stufler.
\newblock Random enriched trees with applications to random graphs.
\newblock Available on-line on \texttt{arXiv:1504.02006}.

\bibitem{tutte1966}
W.~T. Tutte.
\newblock {\em Connectivity in graphs}, volume 285.
\newblock University of Toronto Press, 1966.

\end{thebibliography}

\end{document}